\documentclass{amsart}

\usepackage{latexsym}
\usepackage{booktabs}

\usepackage{amssymb}
\usepackage{amsfonts}
\usepackage{amsmath}
\usepackage{amsthm}

\newtheorem{theorem}{Theorem}
\newtheorem{lemma}{Lemma}[section]

\newtheorem{corollary}[lemma]{Corollary}

\newtheorem{construction}[lemma]{Construction}
\newtheorem{hypotheses}[lemma]{Hypothesis}

\newtheorem{definition}{Definition}[section]
\newtheorem{problem}{Problem}
\newtheorem{notation}{Notation}[section]

\newcommand{\Aut}{{\rm{Aut}}}

\newcommand{\Cay}{{\rm{Cay}}}

\def\ov{\overline}
\def\ovGa{\ov{\Ga}}
\def\ovG{\ov{G}}

\def\lcm{{\rm lcm}}

\def\Ga{\Gamma}

\def\la{\langle}
\def\ra{\rangle}

\def\ZZ{\mathbb{Z}}

\newcommand{\OG}{\mathcal{OG}}

\title[Oriented valency four graphs with cyclic normal quotients]
{Finite edge-transitive oriented graphs of valency four with cyclic normal quotients}

\author[J. A. Al-bar, A. N. Al-kenani,
N. M. Muthana, and 
C. E. Praeger]{Jehan A. Al-bar, Ahmad N. Al-kenani,\\ 
Najat Mohammad Muthana, and 
Cheryl E. Praeger}
\address[All authors]{King Abdulaziz University\\
Jeddah\\
Saudi Arabia} 

\address[Cheryl E. Praeger]
{Also affiliated with: Centre for the Mathematics of Symmetry and Computation\\
School  of Mathematics and Statistics M019\\
The University of Western Australia\\
35 Stirling Highway\\
Crawley, WA 6009\\
Australia}

\thanks{This project was funded by the Deanship of Scientific Research (DSR), King Abdulaziz University, Jeddah, 
under grant no. HiCi/H1433/363-1. The authors, therefore, acknowledge with thanks DSR technical and financial support.
}

\email[Jehan A. Al-bar]{jalbar@kau.edu.sa; jaal[underscore]bar@hotmail.com}
\email[Ahmad N. Al-kenani]{analkenani@kau.edu.sa; aalkenani10@hotmail.com}
\email[Najat M. Muthana]{nmuthana@kau.edu.sa;} 
\email[Najat M. Muthana (second email)]{najat[underscore]muthana@hotmail.com}
\email[Cheryl E. Praeger]{cheryl.praeger@uwa.edu.au}


\begin{document}

\begin{abstract}
We study finite four-valent graphs $\Ga$ admitting an edge-transitive group $G$ of automorphisms such that 
$G$ determines and preserves an edge-orientation on $\Ga$, and such that at least one $G$-normal quotient is a cycle 
(a quotient modulo the orbits of a normal subgroup of $G$). We show on the one hand that the number of distinct 
cyclic $G$-normal quotients can be unboundedly large. On the other hand existence of independent cyclic 
$G$-normal quotients (that is, they are not extendable to a common cyclic $G$-normal quotient) places severe restrictions 
on the graph $\Ga$ and we classify all examples. We show there are five infinite families of such pairs $(\Ga,G)$, 
and in particular that all such graphs have at least one normal quotient which is an unoriented cycle.  
We compare this new approach with existing treatments for the sub-class of weak metacirculant graphs with 
these properties, finding that only two infinite families of examples occur in common from both analyses.
Several open problems are posed.
\end{abstract}

\maketitle

\section{Introduction}\label{sec:intro}

Finite edge-transitive oriented graphs of valency four have been studied intensively because 
of their links to maps on Riemann surfaces \cite{Mar1, MN}. They are simple undirected graphs 
of valency $4$ which admit an orientation of their edges determined and preserved by the action 
of a vertex-transitive and edge-transitive automorphism group. The work of Maru\v{s}i\v{c} 
(summarised in \cite{Mar1}, but see also \cite{Mar2,MP,MW})  demonstrated the importance of a 
certain family of cycles, occurring as subgraphs, for 
understanding the internal structure of these graphs. 
This approach has been exploited recently by Maru\v{s}i\v{c}  and \v{S}parl~\cite{MS} 
for the sub-family of weak metacirculants, to give a classification scheme for this sub-family. 
On the other hand, pursuing a different approach, recent work of the authors in 
\cite{janc1} suggests that cycles occurring as normal quotients play a special role.  
In this paper we study these graphs which have a cycle (oriented or unoriented) as a normal quotient,
and classify the graphs with at least two `independent'
cyclic normal quotients (as defined in Definition~\ref{def:indept}).   
We compare this classification with the analysis in~\cite{MS} for weak metacirculants in
Subsection~\ref{sub:wm}.  

\medskip\noindent
{\bf Normal quotients and basic graphs of cycle type:}\quad 
Let $\OG(4)$ denote the family of all pairs $(\Ga,G)$, 
where $\Ga$ is a finite simple connected undirected graph of valency $4$, and 
$G\leq\Aut(\Ga)$ is a vertex-transitive and edge-transitive group of automorphisms
with a specified $G$-orbit $\Delta$ on ordered vertex-pairs consisting of one ordered 
pair for each edge. In the literature a $G$-action with these properties is said to 
be \emph{$\frac{1}{2}$-transitive}.
An edge orientation is defined as follows: orient each edge $\{x,y\}$ of $\Ga$ from $x$ to $y$
if $(x,y)\in\Delta$. Then $\Ga$ is said to be \emph{$G$-oriented}.
Our notation suppresses the orbit $\Delta$, and we note that the edge-orientation is
determined by $G$ up to possibly replacing $\Delta$ by $\{(x,y)\mid (y,x)\in\Delta\}$
(which corresponds to reversing the orientation of each edge).

For $(\Ga,G)\in\OG(4)$ with vertex set $X$, and a normal subgroup
$N$ of $G$, the \emph{normal quotient} $\Ga_N$ of $(\Ga,G)$ has as
vertices the $N$-orbits in $X$, and a pair $\{B,C\}$ 
of distinct $N$-orbits forms an edge of $\Ga_N$ if and only if  
there is at least one edge $\{x,y\}$ of $\Ga$ with $x\in B$ and 
$y\in C$. 
The quotient $\Ga_N$ is 
\emph{proper} if $N\ne1$ (so that $\Ga_N$ is strictly smaller than $\Ga$). 
There is a constant $\ell$, independent of the 
adjacent pair $B, C$, such that each vertex of $B$ is joined by an 
edge in $\Ga$ to exactly $\ell$ vertices of $C$, \cite[Proposition 3.1]{janc1}. If $\ell=1$ 
then the kernel of the $G$-action on $\Ga_N$ is semiregular and hence equal to $N$,
and we say that $(\Ga,G)$  is  a \emph{normal cover} of $(\Ga_N, G/N)$. 

If all edges from $B$ to $C$ have the same $G$-orientation, then $\Ga_N$ inherits 
a $G$-invariant orientation from $\Ga$. For any normal subgroup $N$, by 
\cite[Theorem 1.1]{janc1}, either $(\Ga,G)$ is a normal cover of $(\Ga_N, G/N)$ and
$(\Ga_N, G/N)\in\OG(4)$, or the quotient $\Ga_N$ is degenerate:
consisting of a single vertex (if $N$ is transitive), or a single edge
(if the $N$-orbits form the bipartition of a bipartite graph $\Ga$), or 
$\Ga_N$ is a cycle possibly, but not necessarily, inheriting a $G$-orientation of its edges.  
Thus, apart from these degeneracies, the family $\OG(4)$ is closed under 
the normal quotient operation, and much can be learned from studying its `basic members',
namely those pairs $(\Ga,G)\in\OG(4)$ for which all proper normal quotients
are degenerate. A study of these basic pairs was initiated in \cite{janc1}. 

In this paper we focus on pairs $(\Ga,G)\in\OG(4)$
with at least one cyclic normal quotient.
The alternative approach in~\cite{MS} mentioned above also uses certain
graph quotients to subdivide the weak metacirculants in $\OG(4)$ into four
broad classes. Our discussion in  Subsection~\ref{sub:wm} shows that the
graph quotients corresponding to two of these four classes  can never 
arise as normal quotients in our sense, and only those for `Class I' of~\cite{MS}
can be cyclic normal quotients of pairs in $\OG(4)$. Moreover,
because of the  major focus in~\cite{MS} on \emph{$\frac{1}{2}$-arc-transitive} 
graphs $\Ga$ (namely those for which $(\Ga,\Aut(\Ga))\in\OG(4)$), 
fewer examples are found than in our situation. We make 
more detailed remarks in Subsection~\ref{sub:wm} (in particular giving definitions
of these concepts).  

\medskip\noindent
{\bf Questions and results:}\quad 
There are many natural questions that arise: how many different 
cyclic  normal quotients can a pair $(\Ga, G)\in\OG(4)$ possess? In particular, 
how many if $(\Ga,G)$ is basic? Can a (basic) pair
$(\Ga,G)$ have  cyclic normal quotients, some of which inherit an edge orientation
while others do not, and if so 
are there restrictions on the numbers of them?
We begin by gathering a few observations concerning these questions 
in Theorem~\ref{thm1}.  
We denote by $C_r$ a cycle of length $r$; and we say that a cyclic normal 
quotient is \emph{$G$-oriented} or \emph{$G$-unoriented} if it
does, or does not, inherit a $G$-invariant edge orientation, respectively.

\begin{theorem}\label{thm1}
Given a positive integer $n$, there exists

\begin{enumerate}
 \item[(a)]   a \textbf{basic} pair $(\Ga, G)\in\OG(4)$
with \textbf{exactly} $n$  pairwise non-isomorphic normal quotients, all $G$-oriented cycles; 
\item[(b)]  a \textbf{basic} pair $(\Ga, G)\in\OG(4)$
with \textbf{at least} $n$ normal quotients, all $G$-oriented cycles of pairwise coprime lengths;
\item[(c)]   a pair $(\Ga, G)\in\OG(4)$
with \textbf{at least} $2n$ normal quotients $\Ga_{N_i}\cong C_{r_i}$ and $\Ga_{M_i}\cong C_{s_i}$, 
for $1\leq i\leq n$, such that $\gcd(r_i,r_j)=\gcd(s_i,s_j)= 1$ for $i\ne j$, and such that each 
$\Ga_{N_i}$ is $G$-oriented while each $\Ga_{M_i}$ is $G$-unoriented. 
\end{enumerate}
\end{theorem}

The pairs $(\Ga, G)$ used in the proof of Theorem\,\ref{thm1}\,(c) 
for $n>1$ are not basic and we do not know of any basic pairs with
these properties.

\begin{problem}\label{problem1}
{\rm
Decide if the number of unoriented cyclic normal quotients of a \emph{basic} pair $(\Ga,G)\in\OG(4)$
can be unboundedly large.
}
\end{problem}


The multitude of cyclic normal quotients in the examples examined to prove
Theorem~\ref{thm1} are all quotients of one or two
particular cyclic normal quotients.  For example, in the proof of Theorem~\ref{thm1}~(c) 
using $(\Ga,G)$ from Construction~\ref{con1}, the quotients 
$\Ga_{N_i}$ are all quotients of a single normal quotient 
$\Ga_{N}\cong C_{r}$ with $r$ divisible by all the $r_i$, and all the $\Ga_{M_j}$
are quotients of a single normal quotient 
$\Ga_{M}\cong C_{s}$ with $s$ divisible by all the $s_j$.

\begin{definition}\label{def:indept}{\rm
Two cyclic normal quotients $\Ga_M, \Ga_N$ of $(\Ga, G)\in\OG(4)$ are \emph{independent}
if the normal quotient $\Ga_K$ is not a cycle, where $K = \tilde N\cap\tilde M$, with 
$\tilde N, \tilde M$ the subgroups consisting of all the elements of $G$ which fix 
setwise each $N$-orbit, or each $M$-orbit, respectively. 
}
\end{definition}

If $\Ga_M, \Ga_N$ are cyclic normal quotients of $(\Ga, G)\in\OG(4)$ and the normal quotient 
$\Ga_K$ is a cycle, where $K = \tilde N\cap\tilde M$,  
then both $\Ga_N$ and $\Ga_M$ are isomorphic to $(G/K)$-normal quotients of 
$\Ga_K$, and hence both or neither of them are $G$-oriented, according as $\Ga_K$
is $G$-oriented or not. Thus if one of $\Ga_N, \Ga_M$ is $G$-oriented and the other is $G$-unoriented, then  
$\Ga_N, \Ga_M$ must be independent. The graphs used to prove Theorem~\ref{thm1}(a), (b) do not have independent cyclic
normal quotients, but those used to prove Theorem~\ref{thm1}(c) do have such quotients.
Our main result classifies all pairs $(\Ga,G)\in\OG(4)$ with independent cyclic normal quotients of given orders.


\begin{theorem}\label{thm3}
Let $(\Ga,G)\in\OG(4)$, let $x$ be a vertex, and suppose that $(\Ga,G)$ has independent
cyclic normal quotients $\Ga_N \cong C_r$ and $\Ga_M \cong C_s$, where $r\geq3, s\geq3$. 
Then $G_x\cong Z_2$, and the following hold: 
\begin{enumerate}
 \item[(a)] at least one of $\Ga_N, \Ga_M$ is $G$-unoriented, say $\Ga_N$ is $G$-unoriented;

\item[(b)] $(\Ga,G)\in\OG(4)$ is a normal cover of $(\ovGa,\ov G)\in\OG(4)$,
which has independent cyclic normal quotients $\ovGa_{\ov N}\cong C_r$ and $\ovGa_{\ov M}\cong C_s$
such that $\ov N\cap\ov M=1$;

\item[(c)] $\ovGa, \ov G$ are as in one of the lines of Table~$\ref{tbl1}$, 
and $\ovGa_{\ov M}$ (and $\Ga_M$) are $G$-oriented if and only if the entry in column $3$ is `yes'.
\end{enumerate}
\end{theorem}

\begin{table}\begin{center}
\begin{tabular}{lllll}
$\ovGa$      & $\ov G$    & $\ovGa_{\ov M}$? & Conditions on $r$, $s$ & Reference \\ \hline
$\Ga(r,s)$   & $G(r,s)$   & yes         & at least one odd       & Con.~\ref{con1} \& Lems.~\ref{lem:con1},~\ref{lem:t1}\\
$\Ga^+(r,s)$ & $G^+(r,s)$ & yes         & both even              & Con.~\ref{con1} \& Lems.~\ref{lem:con1},~\ref{lem:t2}\\
$\Ga(r,s)$   & $H(r,s)$   & no           & $r$ odd, $s$ even      & Con.~\ref{con2c} \& Lems.~\ref{lem:con2c},~\ref{lem:iv}\\
$\Ga(s,r)$   & $H(s,r)$   & no           & $r$ even, $s$ odd      & Con.~\ref{con2c} \& Lems.~\ref{lem:con2c},~\ref{lem:iv}\\ 
$\Ga^+(r,s)$ & $H^+(r,s)$ & no         & both even              & Con.~\ref{con2c} \& Lems.~\ref{lem:con2c},~\ref{lem:ii}\\ 
$\Ga_2(r,s)$ & $G_2(r,s)$ & no           & both odd               & Con.~\ref{con2a} \& Lems.~\ref{lem:con2a},~\ref{lem:i}\\ \hline
              \end{tabular}
              \end{center}
\caption{Table for Theorem~\ref{thm3}(c)}\label{tbl1}
\end{table}

We give several constructions of families in $\OG(4)$ in Section~\ref{sec:ex}, and prove Theorem~\ref{thm1}.
Then we prove Theorem~\ref{thm3} in Section~\ref{sec:thm3}. 

\medskip\noindent
{\bf Remark}\quad 
(a)\quad Some of the pairs in Theorem~\ref{thm3}(c) have larger cyclic normal quotients, for example,
if $(\ovGa, \ov G)=(\Ga(r,s), G(r,s))$ is as in line 1 of Table~\ref{tbl1} with $s$ even, then 
$\ovGa_{\ov N}=C_r$ extends to $\ovGa_{\ov N_2}=C_{2r}$ (see Lemma~\ref{lem:con1}). 
In this case $\ovGa_{\ov N_2}, \ovGa_{\ov M}$ are independent for $(\ovGa, \ov G)$, and $\ov N_2\cap \ov M = 1$. 
Applying Theorem~\ref{thm3} to  $(\ovGa, \ov G)$ with these cyclic normal quotients of orders $2r$ and $s$, we find that 
the pair $(\ovGa, \ov G)$ is given by line 2 of Table~\ref{tbl1}. Thus $(\Ga(r,s), G(r,s))\cong (\Ga^+(2r,s), G^+(2r,s))$
when $r$ is odd and $s$ is even.  

\medskip\noindent
(b)\quad  For each of the pairs $(\ovGa, \ov G)$ in Theorem~\ref{thm3},
$\ovGa$ is arc-transitive (see Definition~\ref{def1} and Construction~\ref{con2a}).
However there may still be $\frac{1}{2}$-arc-transitive graphs $\Ga$ 
with automorphism groups $G$ such that $(\Ga,G)\in\OG(4)$ has
independent cyclic normal quotients $\Ga_M, \Ga_N$, since 
the normal quotient $\ovGa = \Ga_{M\cap N}$ may admit certain automorphisms
making it arc-transitive which do not lift to automorphisms of $\Ga$. 
In the next subsection we discuss this possibility further.

\begin{problem}\label{prob2}{\rm
Describe the maximal cyclic normal quotients of all the pairs 
$(\ovGa, \ov G)$ in Table~\ref{tbl1}, and in particular decide whether there are 
any further relations between these graph families (beyond the
isomorphisms given in the remark above.
}
\end{problem}

\subsection{Comparison with the approach of Maru\v{s}i\v{c}  
and \v{S}parl~\cite{MS} for weak metacirculants}\label{sub:wm}  

The first infinite family of $\frac{1}{2}$-arc transitive graphs,
constructed in~\cite{AMN}, consisted of `metacirculant graphs' 
which had been introduced by Alspach and Parsons~\cite{AP} in 1982.
Several other constructions of $\frac{1}{2}$-arc transitive graphs 
also turned out to be metacirculants. 
Recently Maru\v{s}i\v{c} and \v{S}parl~\cite{MS} embarked on a
thorough analysis of a (proper) sub-family
of $\OG(4)$, called weak metacirculants, with a special 
focus on those which are $\frac{1}{2}$-arc transitive.     

Maru\v{s}i\v{c} and \v{S}parl~\cite{MS} in \cite{MS} define a graph $\Ga$ to be a 
\emph{weak $(m,n)$-metacirculant} relative to an ordered pair 
$(\rho, \lambda)$ of its automorphisms, if $\Ga$ has $mn$ vertices, 
the automorphism $\rho$ has $m$ cycles of length $n$ on vertices, 
the cyclic subgroup $\la \lambda\ra$ permutes the $\rho$-cycles 
transitively, and
$
\rho^\lambda=\rho^r, \ \mbox{for some $r$ such that  $\gcd(r,n)=1$.}
$
The subgroup $\la \rho, \lambda \ra$ is thus a metacyclic group 
which acts transitively on the vertices of $\Gamma$.
If, in addition, $\lambda^m$ fixes a vertex (in which case $\lambda^m$ fixes 
a vertex in each $\rho$-cycle), then $\Ga$ is called a \emph{metacirculant}.
A graph may be a weak metacirculant relative to more than one pair
$(\rho, \lambda)$ (see Construction~\ref{ex4}), and according to \cite[p.\,368]{MS}, it is an open question whether or not
all weak metacirculants are in fact metacirculants (relative perhaps to 
some other pair of automorphisms). 

We classified in Theorem~\ref{thm3} the graph-group pairs $(\Ga,G)\in\OG(4)$ which have
independent cyclic normal quotients. 
We wish to understand which of them are weak metacirculants
relative to pairs of automorphisms lying in $G$.
We hope thereby to gain a better understanding of how our normal quotient analysis 
compares with the analysis in \cite{MS} which focuses on alternating 
cycles. As a bsais for this discussion we make the following assumptions.

\begin{hypotheses}\label{hyp1} 
Let $(\Ga,G)\in\OG(4)$, and suppose that $\Ga$ is a weak $(m,n)$-meta\-circ\-ulant
relative to $(\rho, \lambda)$, for some $\rho,\lambda\in G$. 
Let  $H :=  \la \rho, \lambda \ra$ and $R := \la \rho \ra$, 
so $H\leq G$, the subgroup $R$ is normal in $H$, and $H$ is  transitive 
on the vertices of $\Ga$. 
\end{hypotheses}
\noindent
Since the focus in \cite{MS} is on
$\frac{1}{2}$-arc transitive graphs, and since no Cayley graph 
of an abelian group is $\frac{1}{2}$-arc transitive (see \cite[beginning of \S3]{MS}), 
the authors there assume that $H$ is non-abelian. From this they deduce 
that the number $m$ of $R$-orbits is at least $3$ (\cite[Proposition 3.2]{MS}). 
We lose little generality if we also assume that  $m\geq3$.  
The subdivision  introduced in \cite{MS} of the family of weak 
metacirculants in $\OG(4)$, and studied further in \cite{AS,SI,SII,SIII,SIV},
is defined according to the nature of a certain quotient graph defined modulo 
the $R$-orbits $X_i$ ($1\leq i\leq m$) on vertices (that is to say, the $\rho$-cycles).
 This quotient is slightly different from our graph quotients in that 
it also encodes, for a vertex $x$ in $X_i$, 
the number of edges from $x$ to vertices in $X_j$, for each $j$. Four different 
kinds of quotients are identified in \cite{MS}, and these are used to subdivide 
the weak metacirculants in $\OG(4)$ into four classes, denoted I, 
II, III, IV.
The reason that four different behaviours are observed in \cite{MS} is that 
some of the quotient graphs in \cite{MS} do not correspond to a normal quotient of 
$(\Ga,G)\in\OG(4)$ for any $G$. 
In fact, it is only the quotients arising for graphs 
in Class I of \cite{MS} which can possibly occur as cyclic normal 
quotients in our sense (see Lemma~\ref{lem:classI}). Moreover, 
for $(\Ga,G)\in\OG(4)$, it is possible to have different 
choices of $\rho, \lambda \in G$ leading to different quotients $\Ga_R$ of 
$\Gamma$ which may or may not be normal quotients of $(\Ga,G)$, and if they are
normal quotients, then they may or may not be $G$-oriented.  In Construction~\ref{ex4} we 
give explicit examples of metacirculants $(\Ga, G)\in\OG(4)$ with different pairs 
$(\rho, \lambda)$ of elements in $G$ illustrating each of these possibilities.

Now we turn to the property of having independent cyclic 
normal quotients, which is characterised in Theorem~\ref{thm3}.  
In that result we do not assume a priori that the edge-transitive group contains 
a weak metacirculant subgroup. However, all of the pairs $(\ov{\Ga},\ov{G})$
in the outcome of Theorem~\ref{thm3} turn out to be metacirculants. We
obtain more  Class I weak metacirculants 
in Theorem~\ref{thm3} than those obtained in \cite{MS} because 
our assumptions are valid for some arc-transitive graphs as well as
$\frac{1}{2}$-arc transitive graphs. For example, 
if $\Ga$ is $\frac{1}{2}$-arc transitive, then 
it follows from the proofs of \cite[Lemmas 4.2 and 4.3]{MS}
that the quotient $\Ga_R$ is $G$-oriented, whereas 
the graphs in lines 3--6 of Table~\ref{tbl1}
correspond to examples in Theorem~\ref{thm3}(c) having
a $G$-unoriented normal quotient which can occur as $\Ga_R$
for suitable choices of $\rho, \lambda$.  

Our approach can bring additional insights to the work in \cite{MS}.
In the light of our discussion
it makes sense to restrict to the Class I weak metacirculants for which 
$\Ga_R$ is a cyclic $G$-oriented normal quotient of $(\Ga,G)$. 
Using Theorem~\ref{thm3} we find all such graphs 
with independent cyclic normal quotients.

\begin{corollary}\label{cor:wm} 
Suppose that Hypothesis~$\ref{hyp1}$ holds with $\Ga_R$ a cyclic, 
$G$-oriented, normal quotient of $(\Ga,G)$ of length at least $3$. Suppose also that 
$(\Ga,G)$ has independent cyclic normal
quotients $\Ga_N\cong C_r, \Ga_M\cong  C_s$, as in Theorem~$\ref{thm3}$. 
Then $(\Ga,G)$ is a normal cover of 
$(\ovGa,\ov G)\in\OG(4)$ such that $\ovGa$ is a weak $(\ov{m}, \ov{n})$-metacirculant relative to 
$(\ov{\rho}, \ov{\lambda})$ and one of the lines of Table~$\ref{tbl:wm}$ holds. 
\end{corollary}

\begin{table}\begin{center}
\begin{tabular}{llllllll}
$\ov{m}$ & $\ov{n}$   & $\ovGa$      & $\ov G$    & $\ov{\rho}$   & $\ov{\lambda}$  & Conditions on $r$, $s$& Name for $\ovGa$ in \cite{MS}  \\ \hline
$s$ & $r$   &$\Ga(r,s)$    & $G(r,s)$   & $\mu$    & $\nu$     &at least one odd       & $ X_o(s,r;1)$  \\
$s$ & $r/2$ &$\Ga^+(r,s)$  & $G^+(r,s)$ & $\mu^2$  & $\mu\nu$  &both even              & $X_e(s,r/2;1,0)$\\
 \hline
              \end{tabular}
              \end{center}
\caption{Table for Corollary~\ref{cor:wm}}\label{tbl:wm}
\end{table}

This corollary is proved in Section~\ref{sec:wm}. 
The Class I metacirculant graphs which are $\frac{1}{2}$-arc transitive 
were proved in \cite[Theorem 4.1]{MS} to be precisely the connected 
4-valent tightly attached $\frac{1}{2}$-arc transitive graphs (see also \cite{SII}). 
To assist in comparing Corollary~\ref{cor:wm}  with that classification 
we give in Table~\ref{tbl:wm} also the names of the graphs $\ovGa$
used in \cite[Examples 2.1 and 2.2]{MS}.

\section{Constructions and proof the Theorem~\ref{thm1}}\label{sec:ex}

In this section we examine several infinite families of graph--group pairs in $\OG(4)$, 
and describe their cyclic normal quotients. In Subsection~\ref{sub:thm1} we prove Theorem~\ref{thm1}.

\subsection{Notation}\label{sub:notation} 

For fundamental graph theoretic concepts please refer to the book \cite{GR}. 
For a subset $Y$ of the vertex set of a graph $\Ga$, the \emph{induced subgraph}  $[Y]$ is the
graph with vertex set $Y$ and the edges of $[Y]$ are those edges of $\Ga$ with both vertices in $Y$.   
An induced subgraph will inherit an edge-orientation from an edge-orientation of $\Ga$.

A permutation group $N$ on a set $V$ is
\emph{semiregular} if the only element of $N$ fixing a point of $V$ is the identity; also $N$
is \emph{regular} if it is both transitive and semiregular. 

For $(\Ga,G)\in\OG(4)$, if
an edge $\{x,y\}$ is oriented from $x$ to $y$, then
we call $y$ an \emph{out-neighbour} of $x$, and  we call $x$ an \emph{in-neighbour} of $y$. 
By a \emph{neighbour} of $x$ we mean an in-neighbour or an out-neighbour.

We say that graph--group pairs
$(\Ga,G)$ and $(\Ga',G')$ are \emph{isomorphic} if there exist a graph isomorphism $f$ from $\Ga$ to $\Ga'$ 
and a group isomorphism $\varphi:G\rightarrow G'$ such that, for all $x\in V\Ga$ and $g\in G$, $(x^g)f=(xf)^{g\varphi}$.

For a group $K$ with an inverse-closed generating set $S$ such that $1_K\not\in S$ (that is $S^{-1}=\{s^{-1} | s\in S\}$ is equal to $S$), 
the \emph{Cayley graph} $\Cay(K,S)$ is the graph with vertex set $K$ and edges $\{k,sk\}$ for $k\in K, s\in S$. 
The facts that $S$ generates $K$ and $S$ is inverse-closed imply that $\Cay(K,S)$ is connected and undirected, respectively. 
The definition of adjacency implies that $K$ acts faithfully by right multiplication as a vertex-regular group of 
automorphisms of $\Cay(K,S)$.  Also, $S$ is the set of neighbours of the vertex 
$1_K$, and the subgroup of $\Aut(K)$ leaving $S$ invariant acts naturally as a subgroup of automorphisms
stabilising $1_K$. 

\subsection{Preliminaries on cyclic normal quotients}
We first show that oriented and unoriented cyclic normal quotients of $(\Ga,G)\in\OG(4)$ can be distinguished 
by the action on vertices of the normal subgroup.

\begin{lemma}\label{lem:N}
Let $(\Ga,G)\in\OG(4)$ have a cyclic normal quotient $\Ga_N$, and let $\tilde N$
be the subgroup consisting of all elements of $G$ fixing each $N$-orbit setwise.
\begin{enumerate}
\item[(a)] If $\Ga_N$ is $G$-unoriented, then $\tilde N=N$ and is semiregular on $V\Ga$.
\item[(b)] If $\Ga_N$ is $G$-oriented, then $\tilde N$ contains $G_x$, for each vertex $x$.
\end{enumerate}
\end{lemma}

\begin{proof}
Let $x\in V\Ga$, let $y, y'$ be the out-neighbours of $x$, and let $z, z'$ be the in-neighbours of $x$. Let
$B, B'$ denote the $N$-orbits containing $x, y$ respectively. Suppose that $\Ga_N=C_r$, for some $r\geq3$.

(a) Suppose that $\Ga_N$ is $G$-unoriented. Then $B'$ contains also one of the in-neighbours of $x$, say $z$,
and the other neighbours $y', z'$ lie in a third $N$-orbit distinct from $B, B'$.
Consider the stabiliser ${\tilde N}_x$ of $x$ in $\tilde N$. By the definition of $\tilde N$, 
the subgroup $\tilde N_x$ must fix the $N$-orbit $B'$ 
setwise, and hence must fix the unique out-neighbour $y$ of $x$ it contains. 
Thus ${\tilde N}_x$ fixes each of $y$ and $y'$. Similarly ${\tilde N}_x$ fixes each of $z, z'$. 
It follows from the connectivity of $\Ga$ that ${\tilde N}_x=1$. Thus $\tilde N$ is 
semiregular on $V\Ga$, and in particular $\tilde N=N$.

(b) Now suppose that $\Ga_N$ is $G$-oriented. Then $B'$ contains both out-neighbours $y, y'$ of $x$. 
In this case $G$ induces a cyclic group $Z_r$ on $\Ga_N$, and the setwise stabiliser
$G_B$ of $B$ fixes each $N$-orbit setwise, that is to say, $G_B=\tilde N$.
Since $G_x < G_B$, we have $G_x < \tilde N$.  
\end{proof}

Using this lemma we prove that two oriented cyclic normal quotients cannot be independent. 

\begin{lemma}\label{lem:NM}
Let $(\Ga,G)\in\OG(4)$, and let $\Ga_N=C_r$, $\Ga_M=C_s$, for some $r, s\geq3$, where $N, M$ consist of all 
elements of $G$ fixing setwise each $N$-orbit, or each $M$-orbit, respectively. If both $\Ga_N$ and $\Ga_M$ are $G$-oriented, 
then $\Ga_{N\cap M}=C_t$ is $G$-oriented, for some multiple $t$ of $\lcm\{r,s\}$. In particular $\Ga_N, \Ga_M$ are not independent.
\end{lemma}
 
\begin{proof}
By Lemma~\ref{lem:N}(b), the subgroup $K:=N\cap M$ contains $G_x$, 
for a vertex $x$, and in particular $K\ne 1$. This implies, by 
\cite[Theorem 1.1]{janc1}, that $\Ga_K$ is degenerate, and since 
it has order at least $\lcm\{r,s\}\geq3$, it must be a cycle of 
length a multiple $t$ of $\lcm\{r,s\}$. Since $\Ga_N$ is isomorphic 
to a quotient of $\Ga_K$, it follows that $\Ga_K$ is $G$-oriented.  
\end{proof}

\subsection{Examples with many cyclic normal quotients}
The first family of graphs we consider consists of the lexicographic products
$C_r[2.K_1]$ with a natural orientation on their edges.

\begin{construction}\label{ex:najat}
{\rm
Let $r\geq3$ and let $\Ga$ be the graph with vertex set $X=\mathbb{Z}_r\times \mathbb{Z}_2$ 
and edges $\{(i,j),(i+1,j')\}$ for all $i\in\mathbb{Z}_r, j,j'\in\mathbb{Z}_2$, 
that is, $\Ga= C_r[2.K_1]$, the lexicographic product of $C_r$ and $2.K_1$.
We orient the edges so that $(i,j) \rightarrow (i+1,j')$, for all $i,j,j'$.
Let 
\[
G=Z_2\wr Z_r=\{(\sigma_1,\dots,\sigma_r)\tau^\ell\, |\, 0\leq \ell<r,\ \mbox{each}\ \sigma_k\in \ZZ_2 \} 
\] 
where $(\sigma_1,\dots,\sigma_r):(i,j)\mapsto (i, j+\sigma_i)$, and $\tau:(i,j)\mapsto (i+1,j)$, and let 
\[
B=\{(\sigma_1,\dots,\sigma_r) |\ \mbox{each}\ 
\sigma_k\in \ZZ_2\}=Z_2^r, 
\]
the `base group' of $G$.
}
\end{construction}

By \cite[Lemma 3.6]{janc1}, $G$ preserves the edge orientation, $(\Ga, G)\in\OG(4)$, and $(\Ga,G)$ is basic of cycle type.
We parametrise all the cyclic normal quotients of $(\Ga,G)$ with the set of divisors of $r$, and show that
$(\Ga,G)$ does not have independent cyclic normal quotients. A cyclic normal quotient $\Ga_N$ is \emph{maximal}
if there is no normal subgroup $K$ of $G$, contained in $N$, and such that $\Ga_K$ is cyclic of order larger than $\Ga_N$.

\begin{lemma}\label{lem:najat}
Let $r, \Ga, G, X$ be as in Construction~$\ref{ex:najat}$, and for a divisor $c$ of $r$
let $N(c)= B\cdot\la\tau^{c}\ra$. Then 
$N(c)$ is normal in $G$ and $\Ga_{N(c)}=C_c$ and is $G$-oriented, if $c\geq3$, or is 
$K_2$ or $K_1$ if $c=2$ or $1$, respectively. 
Moreover each proper 
normal quotient of $(\Ga,G)$ is equal to
$\Ga_{N(c)}$ for some $c$, and $\Ga_B=\Ga_{N(r)}$ is the unique maximal
cyclic normal quotient.
\end{lemma}

\noindent{\it Proof.\ }
It was shown in  \cite[Lemma 3.6]{janc1} that $\Ga_B\cong C_r$ and is $G$-oriented. 
By definition $N(c)$ contains $B$ and, since $G/B\cong Z_r$, the group $N(c)$ 
is normal in $G$ and $G/N(c)\cong Z_c$. Also the graph $\Ga_{N(c)}$ is isomorphic to the quotient of
$(\Ga_B,G/B)$ relative to the normal subgroup $N(c)/B\cong Z_{r/c}$ of $G/B$, 
so $\Ga_{N(c)}\cong C_c$ and is $G$-oriented, if $c\geq3$,
and is $K_2, K_1$ if $c=2, 1$ respectively.
 
Let $N$ be a nontrivial normal subgroup of $G$, and consider $\Ga_N$. It was shown in the proof of 
\cite[Lemma 3.6]{janc1} that each $N$-orbit on vertices is a union of 
some $B$-orbits. Thus $B$ fixes each vertex of $\Ga_N$ and so $\Ga_N=\Ga_{BN}$.
Since $G/B\cong Z_r$ it follows that $BN=N(c)$ for some divisor $c$ of $r$,
and hence $\Ga_N=\Ga_{N(c)}$. If $c=r$ then $N(r)=B$, and we conclude that
$\Ga_B=\Ga_{N(r)}$ is the unique maximal
cyclic normal quotient of $(\Ga,G)$. 

\subsection{Graphs with both independent oriented and unoriented cyclic normal quotients}
To prove the second part of Theorem~\ref{thm1}, we use a family of graphs given in Definition~\ref{def1}.
We define an edge-orientation and a group of automorphisms preserving it. We also describe a certain subgraph 
in some cases.   
These graphs were considered in \cite[Section 3]{Mar2} 
from the point of view of their alternating cycle structure (see also \cite[Example 2.4]{MP}).

Recall the concept induced subgraph from Subsection~\ref{sub:notation}.
Note that, if $k$ is even, then the integers representing a given element of $\ZZ_{k}$
are either all even, or all odd, and hence in this case elements of $\ZZ_k$
have a well-defined parity. First we define the undirected graphs and subgroups of their automorphism groups.

\begin{definition}\label{def1}{\rm
Let $r, s$ be integers, each at least $3$. 
Define the undirected graph $\Ga(r,s)$ to have vertex set $X:=\ZZ_r\times \ZZ_s$, such that a vertex 
$(i,j)\in X$ is joined by an edge to each of the four vertices $(i\pm 1, j\pm 1)$. 
Also, if $r, s$ are both even define 
\[
X^+:=\{(i,j)\in X\,|\, i, j\ \mbox{of the same parity}\} 
\]
and let $\Ga^+(r,s)=[X^+]$, the induced subgraph. 
Finally, define the following permutations of $X$, for $(i,j)\in X$,
\begin{align*}
\mu\,&: (i,j)\mapsto (i+1,j), \quad
&\nu\,&: (i,j)\mapsto (i,j+1), \quad \\
\sigma\,&: (i,j)\mapsto (-i,j), \quad
&\tau\,&: (i,j)\mapsto (-i,-j),
\end{align*}
and define the groups as in Table~\ref{tbl:groups}, where in lines 3 and 4 
($r, s$ both even), we identify $\mu,\nu,\sigma, \tau$ 
with their restrictions to $X^+$, and consider the subgroups $G^+(r,s)$ and $H^+(r,s)$ 
acting on $X^+$.

\begin{table}\begin{center}
\begin{tabular}{llll}
$\Ga$         & $G$  & Generators for $G$ & Conditions on $r$, $s$  \\ \hline
$\Ga(r,s)$   & $G(r,s)$   & $\mu, \nu, \sigma$         &  --       \\
$\Ga(r,s)$   & $H(r,s)$  & $\mu, \sigma\nu, \tau$      &$s$ even        \\
$\Ga^+(r,s)$ & $G^+(r,s)$ & $\mu^2, \mu\nu, \sigma$    & both even         \\
$\Ga^+(r,s)$ & $H^+(r,s)$ & $\mu^2,\sigma\mu\nu,\tau$  & both even          \\ \hline
              \end{tabular}
              \end{center}
\caption{Table for Definition~\ref{def1}}\label{tbl:groups}
\end{table}
 
Also let $\tilde M = \la\mu,\sigma\ra=D_{2r}$, $M = \la\mu\ra=Z_{r}$, 
$N=\la\nu\ra\cong Z_s$, and $N'=\la\nu^2,\tau\sigma\nu\ra$, and note that
$G(r,s)=\tilde M\times N$ and  $H(r,s)=
(M\times N').\la\tau\ra$ (with $s$ even; note that $(\sigma\nu)^2=\nu^2$).

Let $M_t=\la\mu^t\ra$ for $t\mid r$, and $N_t=\la\nu^t\ra$ for $t\mid s$. 
If $r$ and $s$ are both even, we also consider the following subgroups restricted 
to their actions on $X^+$:   
$\tilde M^+ = \la\mu^2,\sigma\ra=D_{r}$, $M^+=M_2=Z_{r/2}$, and $N^+=N_2\cong Z_{s/2}$.
}
\end{definition}

Recall that, for a graph $\Ga$, an action of a group $H$ is $\frac{1}{2}$-transitive on $\Ga$ if $H\leq\Aut(\Ga)$
and $H$ is transitive on the vertices and the edges of $\Ga$, but is not transitive on arcs. 

\begin{lemma}\label{lem:def1}
Let $r, s, \Ga(r,s), G(r,s)$, $H(r,s)$, $M, N$, $\Ga^+(r,s), G^+(r,s), H^+(r,s)$ be as in Definition~$\ref{def1}$. Then
\begin{enumerate}
\item[(a)] $\Ga(r,s)$ is connected if at least one of $r,s$ is odd; while if both $r,s$ are even
then $\Ga(r,s)$ has connected components $[X^+]$ and $[X\setminus X^+]$, and $\Ga^+(r,s)$ is connected.
\item[(b)] $\Ga(r,s)$ is bipartite if and only if at least one of $r,s$ is even; while if both $r, s$ are 
even, then $\Ga^+(r,s)$ is bipartite.  
\item[(c)]  $G(r,s)$, and (if $s$ is even) $H(r,s)$ are $\frac{1}{2}$-transitive on $\Ga(r,s)$; and if $r, s$ are both even, 
then  $G^+(r,s)$ and $H^+(r,s)$ are  $\frac{1}{2}$-transitive on $\Ga^+(r,s)$.
\end{enumerate}
\end{lemma}

\begin{proof}
The $N$-orbits in $X$ are the sets  $B_i:=\{(i,j) \mid j\in\ZZ_s\}$,
for $i\in\ZZ_r$. If $s$ is odd, then by the definition of $\Ga$, for each $i$, 
each vertex of $B_i$ is connected by a path in $\Ga$ to each vertex of $B_{i+1}$, and hence $\Ga$ is connected. 
Suppose then that $s$ is even and define $B_i^{{\rm even}} =\{(i,2j) \mid j\in\ZZ_s\}$ and 
$B_i^{{\rm odd}} = B_i\setminus B_i^{{\rm even}}$, for $i\in\ZZ_r$. 
Then the definition of $\Ga$ implies, for each $i$, 
that each vertex of $B_i^{{\rm odd}}$ (or $B_i^{{\rm even}}$) is connected by a 
path in $\Ga$ to each vertex of $B_{i+1}^{{\rm even}}$ (or $B_{i+1}^{{\rm odd}}$,
respectively). Therefore, if $r$ is odd, $\Ga$ is connected. On the other hand
if $r$ and $s$ are both even, then $\Ga$ is disconnected with connected components $[X^+]$ and $[X\setminus X^+]$.
In this case the induced subgraph $\Ga^+$ is connected. This proves (a).

If $r$ is even, then it follows from the discussion above that $\cup_{i\ \mbox{even}}B_i$
and $\cup_{i\ \mbox{odd}}B_i$ form the parts of a bipartition of $\Ga$, and similarly $\Ga$ 
is bipartite if $s$ is even (the parts of the bipartition being unions of $M$-orbits).
On the other hand if $r,s$ are both odd then $|V\Ga|$ is odd so $\Ga$ is not bipartite. 
If both $r,s$ are even then $\Ga^+$ is bipartite with bipartition $\{(i,j)|i,j\ \mbox{even}\}$,  
$\{(i,j)|i,j\ \mbox{odd}\}$. This proves (b).

It is straightforward to check that each of $\mu, \nu, \sigma, \tau$ preserves 
the edge set of $\Ga=\Ga(r,s)$, so that $G:=G(r,s)$ and $H:=H(r,s)$ lie in $\Aut(\Ga)$. 
Also, if both $r, s$ are even, then $\mu^2,\mu\nu,\sigma, \tau$ all leave $X^+$ invariant, 
so $G^+:=G^+(r,s)$ and $H^+:=H^+(r,s)$ are contained in $\Aut(\Ga^+)$, where $\Ga^+=\Ga^+(r,s)$. 
The subgroup $\la\mu, \nu\ra$ of $G$ acts regularly on $X$ and so also, if $s$ is even, 
does the subgroup $\la\mu,\nu^2,\tau\sigma\nu\ra$ of $H$. Hence $G$ and $H$ are vertex-transitive.
The stabiliser in $G$, or in $H$, of the vertex $x=(0,0)$ is $\la \sigma\ra$, or $\la \tau\ra$, respectively.  
The element $\sigma$ acts on the four neighours of $x$ by 
interchanging $(1,1)$ and $(-1,1)$ and interchanging $(-1,-1)$ and $(1,-1)$, and the element
$\sigma\mu\nu^{-1}\in G$ maps the edge $\{x,(1,1)\}$ to the edge $\{x,(1,-1)\}$. Thus $G$ is edge-transitive but not arc-transitive on $\Ga$.
If $s$ is even, then $\tau$ interchanges  $(1,1)$ and $(-1,-1)$ and interchanges $(1,-1)$ and $(-1,1)$, and the element
$\tau\mu\sigma\nu\in H$ maps the edge $\{x,(1,1)\}$ to the edge $\{x,(-1,1)\}$. So $H$ also is edge-transitive but not arc-transitive on $\Ga$.
Similar arguments show firstly that (i) the subgroup $\la\mu^2,\mu\nu\ra$ of $G^+$ acts regularly on $X^+$, 
that $G^+_x=\la \sigma\ra$, and that $G^+$ is $\frac{1}{2}$-transitive on $\Ga^+$; and secondly that (ii) 
the subgroup $\la\mu^2,\sigma\mu\nu\ra$ of $H^+$ acts regularly on $X^+$, 
that $H^+_x=\la \tau\ra$, and that $H^+$ is $\frac{1}{2}$-transitive on $\Ga^+$. This proves part (c).
 \end{proof}

First we define an edge-orientation on the graphs in Definition~\ref{def1} which leads to independent cyclic
normal quotients, one oriented and the other not. We note that this family of oriented graphs was studied in 
\cite[p.50-51. Props. 3.2, 3.3]{Mar2} and \cite[p.159, Proposition 3.3]{MN}, where they were characterised by
properties of their alternating cycles.

\begin{construction}\label{con1}
{\rm
Let $r,s, \Ga(r,s), \Ga^+(r,s)$ be as in Definition~\ref{def1}.
Define an edge-orientation of $\Ga(r,s)$ such that $(i,j)\rightarrow (i\pm1,j+ 1)$ for each $(i,j)\in X$. 
If $r, s$ are both even this restricts to an edge-orientation of $\Ga^+(r,s)=[X^+]$.
}
\end{construction}

\begin{lemma}\label{lem:con1}
Let $\Ga=\Ga(r,s), G=G(r,s), N, N_2, M,\tilde M$,  $\Ga^+=\Ga^+(r,s), G^+=G^+(r,s), N^+, M^+, \tilde M^+$ be as in Definition~$\ref{def1}$,
with edge-orientations of $\Ga,\ \Ga^+$ as in Construction~$\ref{con1}$. 
\begin{enumerate}
\item[(a)] If at least one of $r, s$ is odd, then $(\Ga,G)\in\OG(4)$ 
and $\Ga_N=C_r, \Ga_{\tilde M}=\Ga_M= C_s$ are independent cyclic normal quotients; 
$\Ga_N$ is $G$-unoriented while $\Ga_M$ is $G$-oriented. 
Moreover if $s$ is even then $\Ga_{N_2}=C_{2r}$ and is $G$-unoriented.  
\item[(b)] If  both $r,s$ are even, then $(\Ga^+,G^+)\in\OG(4)$
and $\Ga^+_{N^+}=C_r, \Ga^+_{\tilde M^+}=\Ga^+_{M^+}=C_s$ are independent cyclic normal quotients;
$\Ga^+_{N^+}$ is $G$-unoriented while $\Ga^+_{M^+}$ is $G$-oriented.
\end{enumerate}
\end{lemma}

\begin{proof} 
By Lemma~\ref{lem:def1}, $G, G^+$ are $\frac{1}{2}$-transitive on $\Ga,\ \Ga^+$  respectively,  
and it is straightforward to check that the edge-orientation of Construction~\ref{con1} is preserved in each case. 
If at least one of $r, s$ is odd, then $\Ga$ is connected, by Lemma~\ref{def1} and hence $(\Ga, G)\in\OG(4)$. 
Similarly if both  $r, s$ are even, then $\Ga^+$ is connected, by Lemma~\ref{def1} and hence $(\Ga^+, G^+)\in\OG(4)$.

Each of $N, M, \tilde M$ is an intransitive normal 
subgroup of $G$. The $M$-orbits in $X$ are $C_j=\{(i,j) \mid i\in\ZZ_r\}$, for $j\in\ZZ_s$, and 
are the same as the orbits of $\tilde M$, and the $N$-orbits in $X$ are the subsets $B_i=\{(i,j) \mid j\in\ZZ_s\}$,
for $i\in\ZZ_r$. Also each of  $N^+, M^+, \tilde M^+$ is a normal  
subgroup of $G^+$ and is intransitive on $X^+$. The $M^+$-orbits in $X^+$ are $C_j^+:=C_j\cap X^+$, 
for $j\in\ZZ_s$, and are the same as the $\tilde M^+$-orbits in $X^+$, and the $N^+$-orbits in $X^+$ 
are the subsets $B_i^+:=B_i\cap X^+$ for $i\in\ZZ_r$.

By the definition of $\Ga$, for each $x=(i,j)\in B_i$, each of $B_{i-1}$ and $B_{i+1}$ contains 
one out-neighbour and one in-neighbour of $x$, and it follows 
(as in \cite[Proposition 3.1(b)]{janc1}) that $\Ga_N$ is a connected $G$-arc transitive graph 
of valency $2$ and order $rs/|B_i|=r$, so $\Ga_N=C_r$ and is $G$-unoriented. Similarly, for each $x=(i,j)\in C_j$, 
both out-neighbours of $x$ lie in $C_{j+1}$ and both in-neighbours of $x$ lie in $C_{j-1}$, so
(as in \cite[Proposition 3.1(a)]{janc1}) $\Ga_M=\Ga_{\tilde M}$ is a connected, $G$-oriented, 
$G$-edge transitive graph of valency $2$ and order $rs/|C_j|=s$, that is, $\Ga_M=C_s$ and is $G$-oriented.
If in addition $s$ is even (so $r$ is odd) then each $N$-orbit $B_i$ is the union of two 
$N_2$-orbits, namely $B_i^{even}=\{(i,j)| j\ \mbox{even}\}$ and
$B_i^{odd}=\{(i,j)| j\ \mbox{odd}\}$. For $j$ even, the vertex $x=(i,j)$ has one out-neighbour in each of 
$B_{i-1}^{odd}$ and $B_{i+1}^{odd}$, and for $j$ odd, $(i,j)$ has one out-neighbour in each of 
$B_{i-1}^{even}$ and $B_{i+1}^{even}$. Since $r$ is odd, it follows that $\Ga_{N_2}$ is the cycle $C_{2r}$ and is $G$-unoriented.  
Part (a) now follows from Definition~\ref{def:indept}, since $N\cap \tilde M=1$. 
Part (b) also follows since the $B_i^+=B_i\cap X^+, C_j^+=C_j\cap X^+$, $\tilde M^+=\tilde M\cap G^+, M^+=M\cap G^+$, $N^+=N\cap G^+$ and $\tilde M^+\cap N^+=1$. 
\end{proof}

\subsection{Graphs with independent unoriented cyclic normal quotients}

We now give two constructions of oriented graphs with independent unoriented 
cyclic normal quotients. 
The graphs in the first construction are the 
graphs $\Ga(r,s)$ and $\Ga^+(r,s)$ of Definition~\ref{def1} with a different edge-orientation 
from that in Construction~\ref{con1}.  

\begin{construction}\label{con2c}
{\rm
Let $r, s$ be positive integers, with $r\geq 3$, and $s$ even, $s\geq4$,
and recall the graph $\Ga(r,s)$ of Definition~\ref{def1} with vertex set $X=\ZZ_r\times\ZZ_s$. 
Define an edge-orientation as follows. 
\begin{align*}
&\mbox{if $j$ is even then}\qquad  &(i,j)\rightarrow (i+1,j+ 1)\quad &\mbox{and}\ 
&(i,j)\rightarrow (i-1,j- 1),\\
&\mbox{if $j$ is odd then}\qquad  &(i,j)\rightarrow (i+1,j- 1)\quad &\mbox{and}\ 
&(i,j)\rightarrow (i-1,j+ 1).
\end{align*}
In other words,
$
\mbox{$(i,j)\rightarrow (i+1,j+ (-1)^j)$\quad and \quad $(i,j)\rightarrow (i-1,j- (-1)^j)$.} 
$
Note that this orientation is 
well-defined since elements of $\ZZ_{s}$ have a well-defined parity  as $s$ is even. 
If $r$ and $s$ are both even, then this edge-orientation restricts to an edge-orientation of $\Ga^+(r,s)=[X^+]$. 
}
\end{construction}

\begin{lemma}\label{lem:con2c}
Let  $s$ be even, $s\geq4$, and let $r\geq3$. Let $\Ga=\Ga(r,s), H=H(r,s)$, $N', M$,
and, if $r$ is even, also $\Ga^+=\Ga^+(r,s), H^+=H^+(r,s), N^+, M^+,$ be as in Definition~$\ref{def1}$,
with edge-orientations of $\Ga,\ \Ga^+$ as in  Construction~$\ref{con2c}$.
\begin{enumerate}
\item[(a)] If $r$ is odd, then $H$ preserves the edge-orientation of $\Ga$, 
$(\Ga,H)\in\OG(4)$, and $\Ga_{N'}=C_r$, $\Ga_M\cong C_s$ are independent, $H$-unoriented, cyclic normal quotients.
\item[(b)] If $r$ is even, then $H^+$ preserves the edge-orientation of $\Ga^+$, 
$(\Ga^+,H^+)\in\OG(4)$, and $\Ga^+_{N^+}=C_r$, $\Ga^+_{M^+}\cong C_s$ are independent, $H^+$-unoriented, cyclic normal quotients. 
\end{enumerate}
\end{lemma}

\begin{proof}
It is easy to check that each of 
$\mu,\nu^2,\tau$ preserves the edge-orientation of $\Ga$, and with a little care, that $\sigma\nu$ does also. Thus $H$ preserves the edge-orientation of $\Ga$, and
$H^+$ preserves the edge-orientation of $\Ga^+$.

(a) Suppose that $r$ is odd. Then by Lemma~\ref{lem:def1}, $\Ga$ is connected and $H$ acts $\frac{1}{2}$-transitively, so $(\Ga,H)\in\OG(4)$.
Now $M$, $N'$ are normal subgroups of $H$. The $N'$-orbits are $B_i=\{(i,j) \mid j\in\ZZ_s\}$, for $i\in\ZZ_r$. 
Let $x'=(i,j)\in B_i$. Then each of $B_{i-1}$ and $B_{i+1}$ contains 
one out-neighbour and one in-neighbour of $x'$. Hence, 
as in \cite[Proposition 3.1(b)]{janc1}, $\Ga_{N'}$ is a connected $H$-arc transitive graph 
of valency $2$ and order $rs/|B_i|=r$, so $\Ga_{N'}\cong C_r$ and is $H$-unoriented.
Similarly the $M$-orbits are 
$C_j=\{(i,j) \mid i\in\ZZ_r\}$, for $j\in\ZZ_s$, and for
each $x'=(i,j)\in C_j$, each of $C_{j-1}$ and $C_{j+1}$ contains 
one out-neighbour and one in-neighbour of $x'$, and so $\Ga_M\cong C_s$  
and is $H$-unoriented. By Lemma~\ref{lem:N}, 
$N', M$ are the kernels of the $H$-actions on $\Ga_{N'}, \Ga_M$, respectively.
Since $N'\cap M=1$, the quotients $\Ga_{N'},\Ga_{M}$ are independent. This proves part (a).

(b) Suppose now that $r$ is even, $r\geq4$. Then by Lemma~\ref{def1}, $\Ga^+$ is connected and $H^+$ acts $\frac{1}{2}$-transitively, so $(\Ga^+,H^+)\in\OG(4)$.
Now $M^+$, $N^+$ are normal subgroups of $H^+$. The $N^+$-orbits in $X^+$ (with $X^+$ as in Definition ~\ref{def1}) are $B_i\cap X^+$, for $i\in\ZZ_r$. 
Let $x'=(i,j)\in B_i\cap X^+$. Then each of $B_{i-1}\cap X^+$ and $B_{i+1}\cap X^+$ contains 
one out-neighbour and one in-neighbour of $x'$. Hence, 
as in \cite[Proposition 3.1(b)]{janc1}, $\Ga^+_{N^+}$ is a connected $H^+$-arc transitive graph 
of valency $2$ and order $|X^+|/|B_i\cap X^+|=r$, so $\Ga^+_{N^+}\cong C_r$ and is $H^+$-unoriented.
Similarly the $M^+$-orbits are 
$C_j\cap X^+$, for $j\in\ZZ_s$, and for
each $x'=(i,j)\in C_j\cap X^+$, each of $C_{j-1}\cap X^+$ and $C_{j+1}\cap X^+$ contains 
one out-neighbour and one in-neighbour of $x'$, and so $\Ga^+_{M^+}\cong C_s$ and is $H^+$-unoriented. By Lemma~\ref{lem:N}, 
$N^+, M^+$ are the kernels of the $H^+$-actions on $\Ga^+_{N^+}, \Ga^+_{M^+}$, respectively.
Since $N^+\cap M^+=1$, the quotients $\Ga^+_{N^+},\Ga^+_{M^+}$ are independent. 
\end{proof}

The graphs in the final construction are standard double covers
of the graphs $\Ga(r,s)$ of Definition~\ref{def1}. 

\begin{definition}\label{def:dc}{\rm
The \emph{standard double cover} of a graph $\Ga$ with vertex set $X$ is the graph $\Ga_2$
with vertex set $X_2=\{x_\delta| x\in X, \delta\in\ZZ_2\}$ such that $\{x_\delta, y_{\delta'}\}$ 
is an edge if and only if $\delta\ne\delta'$ and $\{x,y\}$ is an edge of $\Ga$.
}
\end{definition}

Note that $\Ga_2$ has the same valency as $\Ga$ and twice the number of vertices.

\begin{construction}\label{con2a}
{\rm
Let $r, s$ be positive integers, with $r, s\geq 3$, and 
let $\Ga_2(r,s)$ be the standard double cover of the graph 
$\Ga(r,s)$ of Definition~\ref{def1}, so $\Ga_2(r,s)$ has vertex set $X_2$ as in Definition~\ref{def:dc}. 
Define an orientation on the 
edges of $\Ga_2(r,s)$ as follows. 
\begin{align*}
&  (i,j)_0\rightarrow (i+1,j+ 1)_1\ &\mbox{and}&  &(i,j)_0\rightarrow (i-1,j- 1)_1\\
&  (i,j)_1\rightarrow (i+1,j- 1)_0\ &\mbox{and}&  &(i,j)_1\rightarrow (i-1,j+ 1)_0.
\end{align*}
In other words,
$
\mbox{$(i,j)_\delta\rightarrow (i+1,j+ (-1)^\delta)_{\delta+1}$ and  $(i,j)_\delta\rightarrow (i-1,j- (-1)^\delta)_{\delta+1}$.} 
$
We extend the automorphisms defined in Definition~\ref{def1} to maps on $X_2$ as follows. For $(i,j)_\delta\in X_2$,
\begin{align*}
&\mu\,: (i,j)_\delta\mapsto (i+1,j)_\delta, &\nu &: (i,j)_\delta\mapsto (i,j+1)_\delta, \\ 
&\sigma\,: (i,j)_\delta\mapsto (i,-j)_{\delta+1}, &\tau &: (i,j)_\delta\mapsto (-i,-j)_\delta
\end{align*}
and let $G_2(r,s)=\la\mu,\nu,\sigma, \tau\ra=M\times N$, where 
$M=\la\mu,\sigma\tau\ra\cong D_{2r}$, and $N=\la\nu,\sigma\ra\cong D_{2s}$.
%
}
\end{construction}

\begin{lemma}\label{lem:con2a}
Let $r, s, \Ga=\Ga_2(r,s), G=G_2(r,s)$ be as in Construction~$\ref{con2a}$. Then
$G$ preserves the edge-orientation, 
$\Ga$ is connected if and only if $r, s$ are both odd, and in this case $(\Ga,G)\in\OG(4)$, 
and $\Ga_N=C_r$, $\Ga_M=C_s$   are independent cyclic normal quotients, and each is $G$-unoriented.
\end{lemma}

\begin{proof}
Careful but straightforward checking shows that each of the generators of $G$ 
preserves edges, and that each preserves the edge orientation of $\Ga$. 
It is well known and easily proved that a standard double cover of a graph $\Sigma$ is
connected if and only if $\Sigma$ is connected and not bipartite. It follows therefore from Lemma~\ref{lem:def1}
that $\Ga$ is connected if and only if $r,s$ are both odd. Suppose this is the case. 

The subgroup $\la\mu,\nu,\sigma\ra$ 
is normal in $G$ of index $2$, and is regular on vertices. The stabiliser $G_x$ 
of the vertex $x=(0,0)_0$ is $\la\tau\ra$, which interchanges the out-neighbours $(1,1)_1$ and $(-1,-1)_1$,
and the in-neighbours $(1,-1)_1$ and $(1,-1)_1$ of $x$. Thus $G$ is $\frac{1}{2}$-transitive on $\Ga$
and preserves the edge-orientation so $(\Ga,G)\in\OG(4)$.
 
The $N$-orbits in $X_2$ are the subsets $B_i=\{(i,j)_\delta|j\in\ZZ_s,\delta\in\ZZ_2\}$, for $i\in\ZZ_r$. 
Each vertex in $B_i$ has one out-neighbour and one in-neighbour in $B_{i+\varepsilon}$, for $\varepsilon=\pm 1$.
It follows that $\Ga_N=C_r$ and is $G$-unoriented. Similarly the $M$-orbits in $X_2$ are the subsets 
$C_j=\{(i,j)_\delta|i\in\ZZ_r,\delta\in\ZZ_2\}$, for $j\in\ZZ_s$. 
Each vertex in $C_j$ has one out-neighbour and one in-neighbour in $C_{j+\varepsilon}$, for $\varepsilon=\pm 1$,
and hence $\Ga_M=C_s$ and is $G$-unoriented. By Lemma~\ref{lem:N}, $N, M$ are the kernels of the actions of $G$ on $\Ga_N,\Ga_M$
respectively, and since $M\cap N=1$ it follows that $\Ga_N, \Ga_M$ are independent. 
\end{proof}

\subsection{Proof of Theorem~\ref{thm1}}\label{sub:thm1} Let  $n$ be a positive integer.

\medskip\noindent
(a)\quad  Let $r, \Ga, G$ be as in Construction~\ref{ex:najat}. Choose an odd prime $p$ and set $r=p^n$.
By \cite[Lemma 3.6]{janc1}, $(\Ga,G)$ is basic in $\OG(4)$, and  
by Lemma~\ref{lem:najat}, its normal quotients relative to intransitive, nontrivial normal subgroups of $G$ are 
precisely the $G$-oriented quotients $\Ga_{N(c)}=C_c$, for $c\mid r$ and $c>1$.
There are precisely $n$ possibilities, namely for $c=p^i$ with $i= 1,\ldots, n$.

\medskip\noindent
(b)\quad Again let $r, \Ga, G$ be as in Construction~\ref{ex:najat}, and this time choose $n$ distinct 
odd primes $p_1<p_2<\dots<p_n$ and take $r=\prod_ip_i$. 
Again $(\Ga,G)$ is basic in $\OG(4)$ of cycle type, by \cite[Lemma 3.6]{janc1}, and  
by Lemma~\ref{lem:najat}, we have $G$-oriented normal quotients $\Ga_{N(p_i)}=C_{p_i}$ relative to $N(p_i)$, for $i= 1,\ldots, n$.

\medskip\noindent
(c)\quad  
Choose $2n$ odd primes $p_1<p_2<\dots<p_n$ and $q_1<q_2<\dots<q_n$
(where some $p_i$, $q_j$ may be equal), and take $r=\prod_ip_i$ and $s=\prod_iq_i$ in 
Construction~\ref{con1}. By Lemma~\ref{lem:con1}, $(\Ga,G) \in\OG(4)$, and  
$(\Ga,G)$ has a cyclic $G$-oriented normal quotient $\Ga_{N}=C_{r}$, and a cyclic $G$-unoriented normal quotient $\Ga_M=C_s$. For $i= 1,\ldots, n$, 
consider $N_{p_i}=\la\nu^{p_i}\ra$ and $M_{q_i}=\la\mu^{q_i}\ra$, as in Definition~\ref{def1},  and note that 
$N_{p_i}\times M$ and $N\times M_{q_i}$ are both normal in $G$ and intransitive 
on vertices. The normal quotients $\Ga_{N_{p_i}\times M}$ and $\Ga_{N\times M_{q_i}}$ 
are isomorphic to quotients of $\Ga_M$ and $\Ga_N$, and are in fact isomorphic to $C_{p_i}$ and $C_{q_i}$, respectively,
with the former $G$-oriented and the latter $G$-unoriented.
This completes the proof.

\section{Proof of Theorem~\ref{thm3}}\label{sec:thm3}

In this section we analyse the structure of pairs  $(\Ga,G)\in\OG(4)$ with independent cyclic normal quotients. 
The first lemma yields a proof of parts (a) and (b) of Theorem~\ref{thm3}.  

\begin{lemma}\label{lem:NM2}
Let $(\Ga,G)\in\OG(4)$, and suppose that $\Ga_N$, $\Ga_M$ are independent cyclic normal quotients, where $N, M$ consists of all 
elements of $G$ fixing setwise each $N$-orbit, or each $M$-orbit, respectively. Let $\ovGa=\Ga_{N\cap M}, \ov G = G/(N\cap M)$,
$\ov N=N/(N\cap M)$ and $\ov M=M/(N\cap M)$.
\begin{enumerate}
\item[(a)] Then one of $\Ga_N, \Ga_M$ is $G$-unoriented, say $\Ga_N=C_r$
for some $r\geq3$, and the other  $\Ga_M=C_s$, for some $s\geq3$, may be $G$-oriented or $G$-unoriented.
\item[(b)] The quotient $(\ovGa, \ov G)\in\OG(4)$ has independent cyclic normal quotients
$\ovGa_{\ov N}\cong C_r$ and $\ovGa_{\ov M}\cong \Ga_M$
such that $\ov N\cap\ov M=1$, and $(\Ga,G)$ is a 
normal cover of $(\ovGa, \ov G)$. 
\item[(c)] If $N\cap M=1$, then the map $\varphi:g\mapsto (g^{\Ga_N},g^{\Ga_M})$ 
defines a group monomorphism from $G$ to $D_{2r}\times A$, such that $G\varphi\pi_1= D_{2r}, 
G\varphi\pi_2=A$, where $\pi_1, \pi_2$ are the natural projection maps of $D_{2r}\times A$ on $D_{2r}, A$ respectively, and  $A=D_{2s}$ or $Z_s$ 
according as $\Ga_M$ is $G$-unoriented or $G$-oriented, respectively. 
\end{enumerate}
\end{lemma}

\begin{proof}
Part (a) follows from Lemma~\ref{lem:NM}. By Definition~\ref{def:indept}, $\ovGa=\Ga_{N\cap M}$ is not a cycle, and hence,
by \cite[Theorem 1.1]{janc1},  $(\Ga,G)$ is a 
normal cover of $(\ovGa, \ov G)$ and $(\ovGa, \ov G)\in\OG(4)$. 
Now $\ov N$ and $\ov M$ are normal subgroups of $\ov G$, and by construction 
the corresponding $\ov G$-normal quotients satisfy $\ovGa_{\ov N}\cong \Ga_N$ and $\ovGa_{\ov M}\cong \Ga_M$.
These quotients are independent since $\ov N \cap \ov M = 1$. Thus part (b) is proved.

If $N\cap M=1$ then the map $\varphi$ is a monomorphism 
from $G$ to $\Aut(\Ga_N)\times\Aut(\Ga_M)=D_{2r}\times A$, and $G\varphi\pi_1=G^{\Ga_N}=D_{2r}$, $G\varphi\pi_2=G^{\Ga_M}=A$.  
\end{proof}

Parts (a) and (b) of Theorem~\ref{thm3} follow from Lemma~\ref{lem:NM2}. Also, from part (b), 
$(\Ga,G)$ is a normal cover of $(\ovGa, \ov G)$, and so the order of the vertex stabiliser 
$G_x$ is equal to the order of a stabiliser in $\ovG$ of a vertex of $\ovGa$. 
Thus (even to prove that stabilisers have order 2) it is sufficient to consider the case 
where $N\cap M=1$. We therefore make this assumption from now on. We use the following notation.

\begin{notation}\label{notation1}{\rm 
We assume that $(\Ga,G)\in\OG(4)$, and that $\Ga_N=C_r$ ($G$-unoriented), $\Ga_M=C_s$, and 
$A=\Aut(\Ga_M)$, $\varphi, \pi_1,\pi_2$ are as in Lemma~\ref{lem:NM2}, with $N\cap M=1$.
We use the notation introduced in the proof of Lemma~\ref{lem:N}, so $x\in V\Ga$ has 
out-neighbours $y,y'$ and in-neighbours $z,z'$. We let 
$B = x^N, B^+=y^N, B^-=(y')^N$, the pairwise distinct $N$-orbits containing 
$x, y, y'$, and $C=x^M, C^+=y^M, C^-=(y')^M$, the $M$-orbits containing $x, y, y'$
(so $C^+=C^-$ if $\Ga_M$ is $G$-oriented).
Write
\[
 \Aut(\Ga_N)=\la a,c \,|\, a^r=c^2=1, a^c=a^{-1}\ra\cong D_{2r}
\]
where $a$ maps $B$ to $B^+$ and $c$ fixes $B$. We also write 
\[
 \Aut(\Ga_M)=A = \la b,d \,|\, b^s=d^2=1, b^d=b^{-1}\ra\cong D_{2s}\ \mbox{or}\ \la b \,|\, b^s=1\ra\cong Z_s
\]
according as  $\Ga_M$ is $G$-unoriented or $G$-oriented, respectively, 
and $b$ maps $C$ to $C^+$ and (if $\Ga_M$ is $G$-unoriented) $d$ fixes $C$.
}
\end{notation}

\subsection{Case $\Ga_M$ is $G$-oriented.}\label{sub:DCs}
First we derive information about generators of $G\varphi$.

\begin{lemma}\label{lem:NM3}
Using Notation~$\ref{notation1}$, and assuming that $\Ga_M$ is $G$-oriented, 
\begin{enumerate}
\item[(a)] $G_x=\la h\ra\cong Z_2$ is such that $h\varphi=(c,1)$.
\item[(b)] The following all hold, where either $t=1$, or $t=2$ divides $\gcd(r,s)$:
\[
 |G|=\frac{2rs}{t},\ |V\Ga|=\frac{rs}{t},\ N\varphi=\la(1,b^t)\ra,\ M\varphi=\la(a^t,1),(c,1)\ra, 
\]
and $\ G\varphi=\la(a^i,b), M\varphi\ra$, where either $i=0$ or $i=t-1=1$.
\end{enumerate}
\end{lemma}

\medskip\noindent{\it Proof.\quad} 
Since $G$ is edge-transitive on $\Ga$, $G_x\ne1$. Let $h\in G_x\setminus\{1\}$. Then $h$ fixes 
both $B$ and $C$ setwise and so $h^{\Ga_N}\in\la c\ra$ and $h^{\Ga_M}=1$. Since $h\ne 1$ it follows that $h\varphi = (c,1)$
and $G_x=\la h\ra\cong Z_2$, proving part (a). 

Note that $M$ contains $h$, by Lemma~\ref{lem:N}. Since $M\varphi\leq \la a,c\ra\times 1$, it follows that 
$M\varphi=\la(a^t,1),(c,1)\ra$, for some $t\mid r$. Hence $|M|=\frac{2r}{t}, |G|=\frac{2rs}{t}$, and $|V\Ga|=\frac{rs}{t}$.
Also $N\varphi\leq  1\times \la b\ra$ so $N\varphi=\la (1,b^\ell)\ra$ for some $\ell\mid s$, and we have $|N|=\frac{s}{\ell},
|G|=\frac{2rs}{\ell}$. Therefore $\ell=t$ divides $\gcd(r,s)$. Since $G\varphi\pi_2=\la b\ra$, 
$G$ contains an element $g$ such that $g\varphi\pi_2=b$. All such elements satisfy $g\varphi = (a^ic^\delta,b)$ 
for some $i, \delta$. We may replace $g$ by $gm$ for some $m\in M$, and assume that $\delta=0$ and that $0\leq i < t$. 
It remains to prove that $t\leq2$. To see this, note that $G\varphi$ contains $(a^i,b).(a^i,b)^{(c,1)}=
(a^i,b).(a^{-i},b)=(1,b^2)$, which lies in $N\varphi$. \qed

\medskip
We consider the cases $t=1$ and $t=2$ separately. Recall the concepts of regular, 
Cayley graph, and isomorphism of graph--group pairs from Subsection~\ref{sub:notation}.

\begin{lemma}\label{lem:t1}
Under the assumptions of Lemma~$\ref{lem:NM3}$, if $t=1$ then at least one of $r$, $s$ is odd and
$(\Ga,G)$ is isomorphic to the graph--group pair $(\Ga(r,s),G(r,s))$ in Construction~$\ref{con1}$. 
\end{lemma}

\begin{proof}
Suppose that $t=1$ and identify $G$ with $G\varphi$. Then $|V\Ga|=rs$ and $G = \la a,c\ra\times\la b\ra$. 
Moreover the group $K=\la (a,1)\ra\times \la(1,b)\ra$, is normal in $G$, acts regularly 
on $V\Ga$, and $G$ is the semidirect product $K.G_x$.  By \cite[Remark 4.1 and Lemma 4.2]{janc1},
we may assume that $\Ga=\Cay(K,S_0\cup S_0^{-1})$ for a $2$-element generating set $S_0$ for $K$ 
such that $S_0\cap S_0^{-1}=\emptyset$, and we may identify $x=1_K$, $S_0=\{y,y'\}$, and $S_0^{-1}=\{z,z'\}$.
The group $K$ acts by right multiplication and $G_x=\la(c,1)\ra\leq \Aut(K)$ acts naturally on $V\Ga=K$. 
Now $y=(a^i,b^j)$ for some $i\in\ZZ_r, j\in\ZZ_s$, and so $y'=y^{(c,1)}=(a^{-i},b^{j})$. In
particular $i\ne0$ since $y'\ne y$.

Since $N\leq K$, the $N$-orbits are the cosets $N(a^k,1)$ for $k\in \ZZ_r$, and as in Notation~\ref{notation1}, 
$(a,1)$ maps $B=1^N=N$ to $B^+=y^N=N(a^i,1)$. However $(a,1)$ maps $B$ to $N(a,1)$ and hence $i=1$.
Also $M\cap K=\la(a,1)\ra$ is transitive on each $M$-orbit, and the  $M$-orbits are the cosets $M(1,b^k)$ for $k\in \ZZ_s$. 
As in Notation~\ref{notation1}, 
$(1,b)$ maps $C=1^M=M$ to $C^+=y^M=M(1,b^j)$. However $(1,b)$ maps $C$ to $M(1,b)$ and hence $j=1$. 
Thus $S_0=\{(a,b),(a^{-1},b)\}$, and since $K=\la S_0\ra$, at least one of $r, s$ must be odd.

It follows that each vertex $(a^k,b^\ell)$ has out-neighbours $(a^{k\pm1},b^{\ell+1})$,
and so the map $f:(a^k,b^\ell)\mapsto (k,\ell)$ defines a graph isomorphism from $\Ga$ to the graph $\Ga(r,s)$ of Construction~\ref{con1}.
Also the map $(a,1)\mapsto \mu, (1,b)\mapsto \nu, (c,1)\mapsto \sigma$ extends to an isomorphism $\varphi'$ from $G$ to the group $G(r,s)$ of 
Construction~\ref{con1}, and $f,\varphi'$ define an isomorphism from $(\Ga,G)$ to $(\Ga(r,s),G(r,s))$. 
\end{proof}

\begin{lemma}\label{lem:t2}
Under the assumptions of Lemma~$\ref{lem:NM3}$, if $t=2$ then both $r$ and $s$ are even, 
and $(\Ga,G)$ is isomorphic to the graph--group pair  
$(\Ga^+(r,s),G^+(r,s))$ in Construction~$\ref{con1}$. 
\end{lemma}

\medskip\noindent{\it Proof.\quad} 
Suppose that $t=2$ and identify $G$ with $G\varphi$. Then $r,s$ are both even, 
$|V\Ga|=\frac{rs}{2}$, and $G = \la (a^i,b), (a^2,1), (c,1)\ra$ of order $rs$, 
where $i=0$ or $1$, by Lemma~\ref{lem:NM3}. Moreover since $G\pi_1=\la a,c\ra$ it follows that $i=1$.  
Then the group $K:=\la (a^2,1),(a,b)\ra$ is a subgroup of $G$ of index $2$, and as 
$G_x=\la(c,1)\ra\cong Z_2$, we have $K_x=1$ and $K$ acts regularly on $V\Ga$. 
As in Lemma~\ref{lem:t1}, we may assume that $\Ga=\Cay(K,S_0\cup S_0^{-1})$ for a $2$-element generating set $S_0$ for $K$ 
such that $S_0\cap S_0^{-1}=\emptyset$, and we may identify $x=1_K$, $S_0=\{y,y'\}$, and $S_0^{-1}=\{z,z'\}$.
The group $K$ acts by right multiplication and $G_x=\la(c,1)\ra\leq \Aut(K)$ acts naturally on $V\Ga$. 
Thus $y=(a^j,b^k)$ for some $j\in\ZZ_r, k\in\ZZ_s$ of the same parity, and $y'=y^{(c,1)}=(a^{-j},b^{k})$. 

Since $N\leq K$, the $N$-orbits are the cosets $N(a^\ell,1)$ for even $\ell\in \ZZ_r$, and $N(a^\ell,b)$ for odd $\ell\in \ZZ_r$.
As in Notation~\ref{notation1}, 
$(a,b)\in K$ and maps $B=1^N=N$ to $B^+=y^N$, and we have $y^N=N(a^j,1)$ if $j$ is even and $N(a^j,b)$ if $j$ is odd. 
However $(a,b)$ maps $B$ to $N(a,b)$ and hence $j=1$, and so $k$ is odd (since $j, k$ have the same parity).
Also $M\cap K=\la(a^2,1)\ra$ is transitive on each $M$-orbit, and the  $M$-orbits are therefore the cosets 
$M(1,b^\ell)$ for even $\ell\in \ZZ_s$, and $M(a,b^\ell)$ for odd $\ell\in \ZZ_s$. 
As in Notation~\ref{notation1}, 
$(a,b)$ maps $C=1^M=M$ to $C^+=y^M=M(a,b^k)$. However $(a,b)$ maps $C$ to $M(a,b)$ and hence $k=1$. 
Thus $S_0=\{(a,b),(a^{-1},b)\}$.

It follows that each vertex $(a^k,b^\ell)$ has out-neighbours $(a^{k\pm1},b^{\ell+1})$,
and so the map $f:(a^k,b^\ell)\mapsto (k,\ell)$ defines a graph isomorphism from $\Ga$ to the graph $\Ga^+(r,s)$ of Construction~\ref{con1}.
Also the map $(a,b)\mapsto \mu\nu, (a^2,1)\mapsto \mu^2, (c,1)\mapsto \sigma$ extends to an isomorphism $\varphi'$ from $G$ to the group $G^+(r,s)$ of 
Construction~\ref{con1}, and $f,\varphi'$ define an isomorphism from $(\Ga',G)$ to $(\Ga^+(r,s),G^+(r,s))$. \qed    

\subsection{Case $\Ga_M$ is $G$-unoriented.}\label{sub:Cs}
We identify $G$ with $G\varphi$, and we
first derive a short explicit list of possibilities for $G$.

\begin{lemma}\label{lem:NM4}
Using Notation~$\ref{notation1}$, identify $G$ with $G\varphi$, and assume that $\Ga_M=C_s$ is $G$-unoriented. Then 
$G_x=\la (c,d)\ra\cong Z_2$, and one of the following holds.
\begin{enumerate}
 \item[(i)] $|V\Ga|=2rs$, $M=\la (a,1), (c,1)\ra, N=\la (1,b), (1,d)\ra$, and $G=M\times N$;
 \item[(ii)] $r,s$ are both even, $|V\Ga|=rs$, $M=\la (a^2,1), (c,1)\ra, N=\la (1,b^2), (1,d)\ra$, and $G=\la M, N, (a,b)\ra$;
\item[(iii)] $r,s$ are both even, $|V\Ga|=\frac{rs}{2}$, $M=\la (a^2,1)\ra, N=\la (1,b^2)\ra$, and $G=\la M, N,$ $(a,d^\delta b), (c,d)\ra$,
where $\delta=0$ or $1$;
\item[(iv)] $|V\Ga|=rs$, $G=(M\times N).G_x$ with $M\times N$ regular on $V\Ga$, 
where either $M=\la (a,1)\ra$ or $r$ is even and $M=\la (a^2,1), (ac,1)\ra$, and 
either $N=\la (1,b)\ra$ or $s$ is even and $N=\la (1,b^2), (1,bd)\ra$. 
\end{enumerate}
\end{lemma}

\begin{proof}
By Lemma~\ref{lem:N}, the subgroup of all elements of $G$ fixing setwise each $N$-orbit, or each $M$-orbit,
is equal to $N$, or $M$, respectively, and both $N$ and $M$ are semiregular on $V\Ga$. 
Since $G$ is edge-transitive on $\Ga$, $G_x$ contains an element $h$ that interchanges 
$y$ and $y'$, and hence $h$ interchanges the $N$-orbits $B^+$ and $B^-$. Thus $h^{\Ga_N}=c$ 
and $G_x\cap N = N_x$ has index $2$ in $G_x$. Since $N$ is semiregular it follows that $N_x=1$ and so $G_x=\la h\ra\cong Z_2$. 
Similarly $h$ interchanges $C^+$ and $C^-$ so that  $h^{\Ga_M}=d$. This implies that $h = (c,d)$, 
proving the first assertion. 

Next we study the setwise 
stabiliser $N_C$ of $C=x^M$ in $N$. Since $N\cap M=1$ we have $N_C\cong N_C^{\Ga_M}\leq G_C^{\Ga_M}=\la d\ra$, 
and it follows that $N_C\leq\la(1,d)\ra$. Similarly  $M_B\leq\la(c,1)\ra$.
Note that $(1,d)\in N_C$ if and only if $(c,1)\in M_B$ (since $(c,d)\in G_x$); in particular $N_C=1$ if and only if $M_B=1$. Also we may assume, 
without loss of generality, that $z, z'$ lie in $B^+, B^-$ respectively (see  Notation~$\ref{notation1}$).

Suppose that $(1,d)\in N_C$ and let $x':=x^{(1,d)}$. Note that $x'\ne x$ (since $N$ is semiregular), and that $x'=x^{(c,d)(1,d)}=x^{(c,1)}$.
Thus $x, x'\in x^N\cap x^M=B\cap C$.  Since $N^{\Ga_M}$ is a normal subgroup of $G^{\Ga_M}=\la b,d\ra$ containing $d$, it follows that 
$N^{\Ga_M}$ contains $dd^b=db^{-1}db=b^2$, so $N$ is either $\la(1,b^2),(1,d)\ra$ with $s$ even, or $\la (1,b),(1,d)\ra$. Thus $|G|=2rs$ or $4rs$, and
$|V\Ga|=rs$ or $2rs$, and hence $|M|=r$ or $2r$, respectively. Since $(c,1)\in M$, a similar 
argument identifies the possibilities for $M$. If $|G|=4rs$ then (i) holds. If $|G|=2rs$ then $r,s$, $|V\Ga|$, $M, N$ are as in (ii)
and $M\times N$ has index 2 in $G$. Since $a\in G^{\Ga_N}$, $G$ contains an element of the form $g=(a,g')$ for some $g'$ and, 
adjusting $g$ by an element of $N$, we may assume that $g'=1$ or $b$. Since $M$ does not contain $(a,1)$ it follows that $g=(a,b)$ and (ii) holds.

Suppose now that $N_C=1$, so also $M_B=1$. Then $N^{\Ga_M}$ is a semiregular normal subgroup 
 of $G^{\Ga_M}=\la b,d\ra=D_{2r}$, and  $N^{\Ga_M}$ does not contain $d$, so we have the following possibilities for $N$: either 
$N=\la(1,b^t)\ra$ for some divisor $t$ of $s$, or $s$ is even and 
$N=\la(1,b^2), (1,bd)\ra$ and in this latter case we set $t=1$. Then $|N|=\frac{s}{t}$,  $|G|=|G^{\Ga_N}|.|N|=\frac{2rs}{t}$, $|V\Ga|=\frac{rs}{t}$,
and so $|M|=\frac{r}{t}$. Hence $t$ divides $\gcd(r,s)$ and a similar argument shows that either $M=\la(a^t,1)\ra$, or $r$ is even, $t=1$, and 
$M=\la (a^2,1), (ac,1)\ra$. In all cases $K:=M\times N$ is semiregular on $V\Ga$ (as it does not contain $(c,d)$)
and has $t$ orbits of size $\frac{rs}{t^2}$. Now each $K$-orbit is a disjoint union of $\frac{r}{t}$ orbits of $N$,
and it follows that an $N$-orbit and an $M$-orbit meet in at most one vertex. In particular $B^+\cap C^+=\{y\}$.
This implies that $z\not\in C^+$ (since we are assuming that $z\in B^+$) and 
hence $z\in C^-$ and $z'\in C^+$.  Now $y^K$ contains $C^+=y^M$ and hence contains $z'$,
which implies that $B^-=(y')^N=(z')^N\subset (z')^K=y^K$. Hence $y^K$ contains $\{y,y',z,z'\}$, and it follows, since $K  \unlhd G$,
that all edges from vertices in $x^K$ go to vertices in $y^K$, and since $\Ga$ is connected this implies that $K$ has at most two orbits in $V\Ga$.
Thus $t=1$ or $2$.  If  $t=1$, then $K$ is regular so $G=K.G_x$ as in (iv). 
Thus we may assume that $t=2$. Then $r, s, M, N, |V\Ga|$ are as in (iii) and $K$ has index 4 in $G$. Since $G^{\Ga_N}$ 
contains $a$, $G$ contains an element of the form $g=(a,g')$ with $g'\in\la b,d\ra$. Now $N$ contains $(1,b^2)$ so we may assume that
$g'=d^\delta b^{\delta'}$ where each of $\delta, \delta'$ is 0 or 1. The subgroup $M$ does not contain $(a,1)$, and so the 
element $g'\ne 1$. Also if $g'=d$ then $M$ would contain $g(c,d)= (ac,1)$, which is not the case. Hence $g=(a,b)$ or $(a,db)$, so $G$ is as in (iii).
\end{proof}

Now we analyse cases Lemma~\ref{lem:NM4}(i)--(iv) separately.

\begin{lemma}\label{lem:i}
If Lemma~$\ref{lem:NM4}(i)$ holds then $r$ and $s$ are both odd and $(\Ga,G)\cong(\Ga_2(r,s),G_2(r,s))$, as in Construction~$\ref{con2a}$. 
\end{lemma}

\begin{proof}
Suppose that Lemma~\ref{lem:NM4}(i) holds. Then $K:=\la (a,1), N\ra$ is a regular normal (index 2) subgroup of $G$,
and as in Lemma~\ref{lem:t2},
we may assume that $\Ga=\Cay(K,S_0\cup S_0^{-1})$ for a $2$-element subset $S_0\subseteq K\setminus\{1\}$, 
with $S_0\cap S_0^{-1}=\emptyset$ such that $K=\la S_0\ra$,  $x=1_K$, $S_0=\{y,y'\}$. 
Now $K$ acts by right multiplication, $G_x=\la (c,d)\ra\leq \Aut(K)$ 
acts naturally on $V\Ga$, and  $S_0$ is a $G_x$-orbit.  Thus $y = (a^i,d^\delta b^j)$ for some $i\in\ZZ_r, j\in\ZZ_s, \delta=0$ or $1$.
This means that $y'=y^{(c,d)}=(a^{-i},d^\delta b^{-j})$ and since $S_0=\{y,y'\}$ generates $K$, it follows that $\delta=1$,
$i$ and $j$ are nonzero, and $\gcd(i,r)=1$. 

Since $N\leq K$, the $N$-orbits are the cosets $N(a^k,1)$ for $k\in \ZZ_r$, and as in Notation~\ref{notation1}, 
$(a,1)$ maps $B=1^N=N$ to $B^+=y^N=N(a^i,1)$. However $(a,1)$ maps $B$ to $N(a,1)$ and hence $i=1$. Thus $y=(a,db^j)$.
To determine the orbits of $M$, we note that $(c,1)=(c,d)(1,d)\in M$ maps $u\in K$ to $u^{(c,d)}(1,d)$ and so the $M$-orbit 
containing $(1,b^k)$ is the set product $\la(a,1)\ra \{(1,b^k), (1,db^{k})\}$. By Notation~\ref{notation1},
$(1,b)$ maps $C=1^M$ to $C^+=y^M=\la(a,1)\ra \{(1,b^{j}), (1,db^{j})\}$. However $(1,b)$ maps $C$ to 
$\la(a,1)\ra \{(1,b), (1,db)\}$ and hence $j=1$. Thus $S_0=\{(a,db),(a^{-1},db^{-1})\}$ and since
$S_0$ generates $K$ we see that both $r$ and $s$ must be odd. 

Finally, defining $f:V\Ga\rightarrow X_2$ by
$(a^i,b^j)\mapsto(i,j)_0,\ (a^i,db^j)\mapsto (i,j)_1$, and the group isomorphism 
$\varphi:G\rightarrow G_2(r,s)$ extending $(a,1)\mapsto \mu, (1,b)\mapsto \nu, (c,1)\mapsto \sigma\tau, (1,d)\mapsto \tau$, we obtain
an isomorphism $(f,\varphi)$ from the pair $(\Ga,G)$ to the pair  $(\Ga_2(r,s),G_2(r,s))$ in Construction~\ref{con2a}.
\end{proof}

\begin{lemma}\label{lem:ii}
The case Lemma~$\ref{lem:NM4}(ii)$ leads to no examples, while if Lemma~$\ref{lem:NM4}(iii)$ holds, then $\delta=1$ 
and $(\Ga,G)\cong(\Ga^+(r,s),H^+(r,s))$, as in Construction~$\ref{con2c}$ . 
\end{lemma}

\begin{proof}
Here $r,s$ are both even. The approach is similar to that of the previous lemma. If Lemma~\ref{lem:NM4}(ii) holds, let $K:=\la (a,b), N\ra$, and if
Lemma~\ref{lem:NM4}(iii) holds, let $K:=\la (a,d^\delta b), N\ra$ (where $\delta=0$ or $1$). In either case $K$ is a normal subgroup of $G$ of index 2, 
and $K$ is  regular on vertices (since $K\cap G_x=1$, implying that $G=KG_x$).
Thus we may assume that $\Ga=\Cay(K,S_0\cup S_0^{-1})$ for a $2$-element generating set $S_0$ for $K$ 
such that $S_0\cap S_0^{-1}=\emptyset$, and we may identify $x=1_K$, $S_0=\{y,y'\}$.

Suppose first that Lemma~\ref{lem:NM4}(ii) holds. Then $y = (a^i, b^jd^{\delta'})$, for some $i\in\ZZ_r, j\in\ZZ_s, 
\delta'\in\ZZ_2$, with $i,j$ of the same parity. Now $S_0$ is a $G_x$-orbit so $y'=y^{(c,d)}=(a^{-i},b^{-j}d^{\delta'} )$. Since $S_0\cap S_0^{-1}=\emptyset$, we have $y'\ne y^{-1}$, and  
hence $\delta'=1$.
However $S_0$ is then contained in the proper subgroup $\la (a^2,1),(1,b^2),(a^i,b^jd)\ra$ of 
index 2 in $K$, contradicting the fact that $S_0=\{y,y'\}$ generates $K$.

Thus  Lemma~\ref{lem:NM4}(iii) holds. Suppose first that $\delta=0$. Then as $y\in K$, we have $y=(a^i,b^j)$
for some $i\in\ZZ_r, j\in\ZZ_s$ of the same parity.  This implies, as in the previous paragraph, that 
$y'=y^{(c,d)}=(a^{-i},b^{-j})$ is equal to $y^{-1}$, which is a contradiction. Thus $\delta=1$. If
$y=(a^i,b^j)$ with both $i,j$ even, then we again find  $y'=y^{-1}$, a contradiction. Hence $y=(a^i,db^j)$ with both $i,j$ odd. 
Since $N\leq K$, the $N$-orbits are the cosets $N(a^{2k},1)$ and $N(a^{2k+1},db)$, for $0\leq k<r/2$, and as in Notation~\ref{notation1}, 
the element $(a,db)\in G$ maps $B=1^N=N$ to $B^+=y^N=N(a^i,db^j)=N(a^i,db)$ (since $(1,b^{2})\in N$). However $(a,db)\in K$ acts by right multiplication, 
and maps $B$ to $N(a,db)$. Hence $i=1$ and $y=(a,db^j)$.
Also $M$ lies in $K$, and the $M$-orbits are the cosets $M(1,b^{2k})$ and $M(a, db^{2k+1})$, for $0\leq k<s/2$. 
As in Notation~\ref{notation1}, the element $(a,db)\in G$ maps $C=1^M=M$ to $C^+=y^M=M(a,db^{j})$. 
However $(a,db)\in K$ acts by right multiplication, 
and maps $C$ to $M(a,db)$. Hence $j=1$ and $y=(a,db)$, $y'=y^{(c,d)}=(a^{-1},db^{-1})$.   

Each vertex $w$ of $\Ga=\Cay(K,S_0\cup S_0^{-1})$ has out-neighbours $yw, y'w$, so for $i, j$ even, 
$w=(a^i,b^j)$ has out-neighours $(a^{i+1}, db^{j+1})$ and $(a^{i-1}, db^{j-1})$, while for $i, j$ odd, 
$w=(a^i,db^j)$ has out-neighours $(a^{i+1}, b^{j-1})$ and $(a^{i-1}, b^{j+1})$. 
Thus the map $f:V\Ga\rightarrow X^+$ defined by
$(a^i,b^j)\mapsto(i,j)$ if $i, j$ are both even, and  $(a^i,db^j)\mapsto(i,j)$ if $i, j$ are both odd, determines a graph
isomorphism from $\Ga$ to $\Ga^+(r,s)$, which maps each oriented edge $w\rightarrow yw, w\rightarrow y'w$ of $\Ga$ to an oriented edge of $\Ga^+(r,s)$ with the edge-orientation of Construction~\ref{con2c}. Also the map $(a^2,1)\mapsto \mu^2, (1,b^2)\mapsto \nu^2, (c,d)\mapsto \tau, (a,db)\mapsto \tau\sigma\mu\nu$,
extends to a group isomorphism 
$\varphi:G\rightarrow H^+(r,s)$, and we obtain
an isomorphism $(f,\varphi)$ from the pair $(\Ga,G)$ to the pair  $(\Ga^+(r,s),H^+(r,s))$ in Construction~\ref{con2c}.
\end{proof}

\begin{lemma}\label{lem:iv}
If Lemma~$\ref{lem:NM4}(iv)$ holds then either
\begin{enumerate}
 \item[(a)] $r$ is odd and $s$ is even,  $M=\la (a,1)\ra$, $N=\la (1,b^2), (1,bd)\ra$, and $(\Ga,G)\cong (\Ga(r,s),H(r,s))$ of Construction~$\ref{con2c}$; or 
 \item[(b)] $r$ is even and $s$ is odd,  $M=\la (a^2,1), (ac,1)\ra$, $N=\la (1,b)\ra$, and $(\Ga,G)\cong (\Ga(s,r),H(s,r))$ of Construction~$\ref{con2c}$.
\end{enumerate}
\end{lemma}

\begin{proof}
Suppose that Lemma~\ref{lem:NM4}(iv) holds, so $|M|=r, |N|=s, |V\Ga|=rs$. There are two possibilities for $M$ and two for $N$. 
For any of these $M, N$, the group
$K:=M\times N$ has index $2$ in $G$ and acts regularly on the vertices of $\Ga$. 
Hence we may assume that $\Ga=\Cay(K,S_0\cup S_0^{-1})$ for a $2$-element generating set $S_0$ for $K$ 
such that $S_0\cap S_0^{-1}=\emptyset$, and we may identify $x=1_K$, $S_0=\{y,y'\}$.

Now $y=mn$ for some $m\in M, n\in N$, and $y'$ is the image of $y$ under conjugation by $(c,d)\in G_x$. 
Since $y'\ne y^{-1}$, it is not possible for both $m^{(c,d)}=m^{-1}$ and $n^{(c,d)}=n^{-1}$ to hold.
Therefore $M,N$ are not both cyclic. Interchanging $r$ and $s$ if necessary, we may assume that 
$n^{(c,d)}\ne n^{-1}$, and it follows from Lemma~\ref{lem:NM4}(iv) that $s$ is even and $N=\la (1,b^2), (1,bd)\ra$. 
We will prove that (a) holds. (In the case where $m^{(c,d)}\ne m^{-1}$ it will then follow,
on interchanging $r$ and $s$ in our arguments, that (b) holds.)   

Since $n^{(c,d)}\ne n^{-1}$, the element $n$ satisfies $n=(1,db^j)$ for some $j\in\ZZ_s$, and we have
$y=m(1,db^j)$, $y'=m^{(c,d)}(1,db^{-j})$. 
Now $K$ acts by right multiplication and $M\leq K$, so the $M$-orbits are $C_{2k} := M(1,b^{2k})$ and 
$C_{2k+1} := M(1,db^{2k+1})$ for $0\leq k<\frac{s}{2}$. The vertex $x$ lies in $C=C_0$, and $y$ lies in the image
$C^+$ of $C$ under the action of the element $(c,b)=(1,bd)(c,d)\in G$ (see Notation~\ref{notation1}, 
since this element lies in $G$ and induces $b$ on $\Ga_M$). However $(c,b)$ maps $C$ to $(M(1,bd))^{(c,d)}=M(1,b^{-1}d)=M(1,db)$,
and since this set must contain $y=m(1,db^j)$, we have $j=1$. Thus $y=m(1,db)$.

The $N$-orbits are the subsets $Nu$ for $u\in M$. The vertex $x$ lies in $B=N$,  and $y$ lies in the image
$B^+$ of $B$ under the action of the element $(a,1)$ if $M=\la (a,1)\ra$, or $(a,d) =(ac,1)(c,d)$ if 
$M=\la (a^2,1),(ac,1)\ra$ (see Notation~\ref{notation1}, 
since these elements lie in $G$ and induce $a$ on $\Ga_N$). 
In the second case, the element $(a,d)$ maps $B$ to $(N(ac,1))^{(c,d)}=N(a^{-1}c,1)=N(ca,1)$,
and since this set must contain $y=m(1,db)$ we have $m=(ca,1)$. However this 
implies that $y=(ca,db)$ which has order 2 and hence $y^{-1}=y\in S_0$, a contradiction. 
Therefore $M=\la (a,1)\ra$. The element $(a,1)$ maps $B$ to $B^+=N(a,1)$ and since this set contains 
$y$ it follows that $m=(a,1)$ and $y=(a,db)$, $y'=(a^{-1},db^{-1})$. Now the fact that $\la S_0\ra=K$ 
implies that $r$ must be odd.

It remains to identify the graph-group pair. The oriented edges of $\Ga$ are the pairs $u\rightarrow yu$ 
and $u\rightarrow y'u$ for $u\in K$. An easy computation shows that, for each $i\in\ZZ_r$ and for $0\leq j<\frac{s}{2}$,
\begin{align*}
 (a^i,b^{2j})&\rightarrow (a^{i+1},db^{2j+1}) &\mbox{and}& &(a^i,b^{2j})&\rightarrow&(a^{i-1},db^{2j-1}),\\
 (a^i,db^{2j+1})&\rightarrow (a^{i+1},b^{2j}) &\mbox{and}& &(a^i,db^{2j+1})&\rightarrow&(a^{i-1},b^{2j+2}).
\end{align*}
It follows that the bijection $f: K\rightarrow \ZZ_r\times \ZZ_s$ given by 
$f:(a^i,b^{2j})\mapsto (i,2j)$ and $f:(a^i,db^{2j+1})\mapsto (i,2j+1)$ defines a graph isomorphism 
from $\Ga$ to the graph $\Ga(r,s)$ such that the oriented edges $u\rightarrow yu$ 
and $u\rightarrow y'u$ of $\Ga$ are mapped to oriented edges according to the edge-orientation defined in Construction~\ref{con2c}.
Also the map $\varphi$ given by  
\[
\varphi: (a,1)\mapsto \mu,\ \varphi: (1,b^2)\mapsto\nu^2,\ \varphi: (1,db)\mapsto \tau\sigma\nu,\ \varphi: (c,d)\mapsto \tau 
\]
extends to an isomorphism from $G$ to the group $H(r,s)$ of Construction~\ref{con2c}.
Moreover the pair $(f,\varphi)$ determines an isomorphism from $(\Ga,G)$ to $(\Ga(r,s),H(r,s))$.
\end{proof}

\subsection{Proof of Theorem~\ref{thm3}}\label{sub:thm3}
Let $(\Ga,G)\in\OG(4)$ have independent cyclic normal quotients $\Ga_N, \Ga_M$ of orders $r, s$ respectively.
We may assume that $N, M$ consists of all elements of $G$ which fix setwise each $N$-orbit, or each $M$-orbit,
respectively. Part (a) of Theorem~\ref{thm3} follows from Lemma~\ref{lem:NM2}(a), and we may assume that $\Ga_N$
is $G$-unoriented.
Part (b)  of Theorem~\ref{thm3} follows from Lemma~\ref{lem:NM2}(b), and (as explained just before 
Notation~\ref{notation1}) for the rest of the proof we may 
assume that $N\cap M=1$. If $\Ga_M$ is $G$-oriented then, by Lemmas~\ref{lem:NM3},~\ref{lem:t1}, 
and~\ref{lem:t2}, $|G_x|=2$ and line 1 or 2 of Table~\ref{tbl1} holds.  
If $\Ga_M$ is $G$-unoriented then, by Lemmas~\ref{lem:NM4},~\ref{lem:i},~\ref{lem:ii},  
and~\ref{lem:iv}, $|G_x|=2$ and one of the lines 3, 4, 5 or 6 of Table~\ref{tbl1} holds. This completes the proof.

\section{Weak metacirculants with independent cyclic normal quotients} \label{sec:wm} 

First we see that under Hypothesis~\ref{hyp1},
it is only the graphs in Class I of
\cite{MS} for which the quotient $\Ga_R$, defined as in Section~\ref{sec:intro},
can possibly be a cyclic normal quotient.

\begin{lemma}\label{lem:classI} 
Suppose that Hypothesis~$\ref{hyp1}$ holds.

\begin{enumerate}
\item[(a)] If $\Ga_R$ is a normal quotient of $(\Ga,G)$,
then $\Ga$ lies in Class I or IV of \cite{MS} relative to 
$(\rho,\lambda)$, and if $\Ga_R$  
is cyclic then $\Ga$ lies in Class I.

\item[(b)] Conversely, for $\Ga$ in Class I or IV
 relative to $(\rho,\lambda)$, either $H$ is regular on vertices,
or $\Ga$ lies in Class I and $(\Ga, H) \in\OG(4)$
with $\Ga_R$ an $H$-oriented cyclic normal quotient.  
\end{enumerate}
\end{lemma}

\begin{proof}
(a)  Suppose that $\Ga_R$ is equal to $\Ga_N$ for some
$N \unlhd G$ (that is to say the sets of $R$-orbits and $N$-orbits 
on vertices are the same). 
By \cite[Thm 1.1]{janc1}, there are no edges joining
vertices in the same $R$-orbit, and moreover there is a 
constant $k$ such that, if there is an edge between some vertex in the 
$R$-orbit $X_i$ and some vertex in $X_j$, then each vertex of $X_i$ is 
adjacent to exactly $k$ vertices in $X_j$. These conditions are not 
satisfied for Classes II and III of \cite[Section 3]{MS}, and 
hence only Classes I and IV can
correspond to normal quotients of $(\Ga,G)$. Moreover, since
the quotients for Class IV have valency $4$, it is only the quotients 
for Class I which can arise as cyclic normal quotients. 

(b) Suppose that $\Ga$ is in Class I or IV of \cite{MS}
 relative to $(\rho,\lambda)$, and that $H$ is not regular on vertices.
Then $H_x\ne 1$ for $x$ in the $R$-orbit $X_1$, and since the action induced by 
$H$ on the $R$-orbits is cyclic (induced by $\lambda$), $H_x$ fixes 
each $R$-orbit setwise. 
If  $H_x$ fixed each vertex adjacent to $x$ then, since $\Ga$ 
 is connected, it would follow that $H_x$ would fix all vertices, 
contradicting our assumption that $H_x\ne 1$. Thus $H_x$ 
moves some vertex $y$ adjacent to $x$, say $y\in X_2$,
and hence $X_2$ contains at least two vertices adjacent to $x$. 
For $\Ga$ in Class IV, the four vertices adjacent to $x$ lie in four
 distinct $R$-orbits and hence we conclude that $\Ga$ lies in 
Class I of \cite{MS}, and that two $R$-orbits distinct from 
$X_1$, namely $X_2$ and, say, $X_3$, each 
contain two vertices adjacent to $x$. Thus $H_x$ interchanges the two
vertices in $\Ga(x)\cap X_2$, and since $H$ preserves the
edge-orientation, it follows that $H$ is edge-transitive, so $(\Ga,H)\in\OG(4)$,
with $\Ga_R$  a cyclic normal quotient. Finally the fact that
 $H_x$ interchanges the two
vertices in $\Ga(x)\cap X_2$ implies that $\Ga_R$ is an $H$-oriented cycle.
(This last fact can also be deduced from the proofs of 
\cite[Lemmas 4.2 and 4.3]{MS} although it is not explicitly stated there.)  
\end{proof}

 We give examples to show that, for a given $(\Ga,G)\in\OG(4)$,
 different pairs $(\rho,\lambda)$  can lead to different behaviours 
for the quotient $\Ga_R$, even with the same subgroup $H = \la \rho, \lambda\ra$.
 
 \begin{construction}\label{ex4} {\rm
Let $r$ be an odd integer, $r\geq3$, and let $\Ga=\Ga(r,r)$ and $G=G(r,r)$,
as in line 1 of Table~\ref{tbl:groups}. Then, using Lemma~\ref{lem:con1}
we see that $(\Ga,G)\in\OG(4)$, and that, for $\rho,\lambda$ as in (a), (b) or (c) below,
$\Ga$ is an $(r,r)$-metacirculant relative to $(\rho,\lambda)$, 
$H:=\la\rho,\lambda\ra = \la\mu,\nu\ra$ is regular, and the quotient $\Ga_R$, for
$R:=\la\rho\ra$, is a cycle $C_r$. 

\begin{enumerate}
\item[(a)] Let $\rho=\mu\nu$ and $\lambda=\mu$.  
Then $\Ga_R$ is $G$-oriented but is not a normal quotient of $(\Ga,G)$ 
since $N_G(R)=H$. (The $R$-orbits are $\{ (i+\ell,\ell)\,|\, \ell\in\mathbb{Z}_r \}$,
for $i\in\mathbb{Z}_r$.)

\item[(b)] Let $\rho=\nu$ and $\lambda=\mu$.  
Then $\Ga_R$ is a normal quotient of $(\Ga,G)$ that is 
not $G$-oriented.

\item[(c)] Let $\rho=\mu$ and $\lambda=\nu$.  
Then $\Ga_R$ is a $G$-oriented normal quotient of $(\Ga,G)$.
\end{enumerate}
}
\end{construction}

Now we study the case where $(\Ga,G)$ has independent cyclic normal 
quotients and where $\Ga_R$ is a $G$-oriented normal quotient, so $\Ga$ is of Class I
of \cite{MS}, by Lemma~\ref{lem:classI}.

\begin{lemma}\label{lem:gar}
Suppose that Hypothesis~\ref{hyp1} holds with $\Ga_R \cong C_m$ a 
$G$-oriented normal quotient of $(\Ga,G)$, and that $(\Ga,G)$ has 
independent cyclic normal quotients. Then for a possibly different pair 
$(\rho,\lambda)$ of elements from $H$, the same conditions all hold and 
$(\Ga,G)$ has a possibly different
pair of independent cyclic normal quotients, one of which is $\Ga_R$.  
\end{lemma} 

\begin{proof}
Let $\Ga_M, \Ga_N$ be independent cyclic normal quotients of $(\Ga,G)$, for 
normal subgroups $M, N$ of $G$. We may assume that $M$ (respectively $N$) 
is equal to the kernel of the $G$-action on $\Ga_M$ (respectively $\Ga_N$). 
Let $\ov{\Ga}=\Ga_{M\cap N}, \ov{G}=G/(M\cap N)$, etc., as in Lemma~\ref{lem:NM2}.
Since $\Ga_M, \Ga_N$ are independent, $(\ovGa, \ovG)\in\OG(4)$.
Moreover, by Theorem~\ref{thm3}, the stabiliser $G_x$ of a vertex $x$ has order $2$. 
Thus $G=HG_x$ and $|G:H|\leq 2$, since $H$ is vertex-transitive. 

Since $\Ga_R$ is a normal quotient of $(\Ga,G)$, there is a normal subgroup 
$\hat R$ of  $G$, containing $R$, with the same vertex-orbits as $R$, 
and such that $\hat R$ is  equal to the kernel of the $G$-action on $\Ga_R$.
Moreover, since $\Ga_R=\Ga_{\hat R}$ is $G$-oriented, the group $G/\hat R$ 
induced on $\Ga_R$ is cyclic of order $m$, and is therefore 
equal to the cyclic group induced by $\la \lambda\ra$ on $\Ga_R$, so 
$G=\la \hat R, \lambda\ra$. 

Suppose first that some $R$-orbit is a union of $j$ of the 
$M$-orbits in $V\Ga$, where $j\geq1$. Then each $R$-orbit is a union of $j$ of the 
$M$-orbits, and $\Ga_R$ is isomorphic to a quotient of $\Ga_M$. Since $m\geq3$
and $\Ga_R$ is $G$-oriented, it follows that $\Ga_M$ is also $G$-oriented. 
Thus $G$ induces a regular cyclic group on $\Ga_M$, and as $H$ is 
vertex-transitive, $H$ induces the same group, so $G=\la M, \lambda'\ra$ for 
some element $\lambda'\in H$. Moreover we may choose $\lambda'$ such that 
$\lambda, \lambda'$ induce the same action on $\Ga_R$, and hence 
$\lambda'=\lambda \rho^i$ for some $i$. Also $M\cap H$ is the subgroup 
$\la \rho'\ra$, where $\rho'=\rho^j$, and Hypothesis~\ref{hyp1}
holds with $(\rho',\lambda')$ in place of $(\rho, \lambda)$, and
$M\cap H$ in place of $R$ (noting that $\rho^\lambda = \rho^r$ implies that
$(\rho')^{\lambda'}=(\rho')^r$, and $\gcd(r,n)=1$ implies that $\gcd(r, o(\rho'))=1$). 
Since $\Ga_M$ is $G$-oriented, $M$ contains $G_x$
by Lemma~\ref{lem:N}(b). Hence $M=(M\cap H)G_x$, $\Ga_{M\cap H}=\Ga_M$ (the `new $\Ga_R$'),
and all the assertions hold relative to  $(\rho',\lambda')$, and the
normal quotients $\Ga_N, \Ga_{M\cap H}$.

Thus we may assume that $\hat R$ does not contain $M$. 
Suppose next that $M$ properly contains $\hat R$, that is, 
$\Ga_M$ is a quotient of $\Ga_R$. Then $\Ga_N, \Ga_R$ are independent
cyclic normal quotients and all the assertions hold without changing 
$(\rho,\lambda)$. So now we may assume also that $M$ does not contain $\hat R$.
Let $T:= R\cap M$ and note that $T\leq H$ since $R\leq H$. 
If $\Ga_T$ is not cyclic then  $\Ga_M, \Ga_R$ are independent
cyclic normal quotients and all the assertions hold without changing 
$(\rho,\lambda)$. Assume then that $\Ga_T$ is cyclic. Since
$\Ga_M$ is a quotient of $\Ga_T$, the pair $\Ga_N, \Ga_T$ are 
independent cyclic normal quotients, and as $T\subseteq R$, each $R$-orbit 
is a union of $T$-orbits. Thus the arguments of the previous paragraph 
may be used to replace $(\rho,\lambda)$ by a new pair
$(\rho', \lambda')$ from $H$, and replace $R$ by $T$, and $\Ga_R$ by 
$\Ga_T$, so that all assertions hold for the independent cyclic 
normal quotients $\Ga_N, \Ga_T$.   
\end{proof}

Finally we prove Corollary~\ref{cor:wm}.

\medskip
\noindent\emph{Proof of Corollary~\ref{cor:wm}.}\quad 
Suppose that Hypothesis~\ref{hyp1} holds with $\Ga_R \cong C_m$ a 
$G$-oriented cyclic normal quotient of $(\Ga,G)\in\OG(4)$, where $m\geq3$.
Suppose also that, as in Theorem~\ref{thm3}, $(\Ga,G)$ has independent cyclic normal quotients
$\Ga_N \cong C_r, \Ga_M \cong C_s$, where $N, M$ are the kernels of the $G$-actions on 
$\Ga_N, \Ga_M$ respectively. 
By Theorem~\ref{thm3}, we may assume that $\Ga_N$ is $G$-unoriented, and
by Lemma~\ref{lem:gar}, we may assume that one of 
the independent cyclic normal quotients is $\Ga_R$. (Note that, from the proof of 
Lemma~\ref{lem:gar}, $\Ga_R$ may be replaced by a cyclic $G$-oriented normal 
quotient of order a proper multiple of the original 
$m$ in Hypothesis~\ref{hyp1}, but in the exposition we can, and will, continue to use 
$m$ as the order of $\Ga_R$.) Since
$\Ga_R$ is $G$-oriented, we have $\Ga_R=\Ga_M \cong C_m$, so $m=s$.  

By Theorem~\ref{thm3} again, setting $T:= M\cap N$, $\ovG := G/T$ and
$\ovGa :=\Ga_{T}$, the pair  $(\Ga,G)$ is a normal cover of 
$(\ovGa,\ov G)\in\OG(4)$, and $\ovGa, \ovG$ and $\ovGa_{\ov{M}} \cong \Ga_M \cong C_{\ov{m}}$
are as in line 1 or 2 of Table~\ref{tbl1}, where $\ov{M}=M/T$ and $\ov{m}=m=s$. 
Let the order of $\ovGa$ be $\ov{m}\,\ov{n}$. Then $\ov{n}$ is $r$ or $r/2$ for 
line 1 or 2 of Table~\ref{tbl1}, respectively. Thus, so far we have proved that
$\ov{m}, \ov{n}, \ovGa, \ovG$, and the conditions on $r, s$ are as in the
appropriate line of Table~\ref{tbl:wm}. Finally we prove that $\ovGa$ is a weak $(\ov{m}, \ov{n})$-metacirculant.

 Consider first line 1 of Table~\ref{tbl:wm}.
Then $\ovG=\la \mu,\nu,\sigma\ra$ where $\mu, \nu, \sigma$, are as in Definition~\ref{def1}.
It is easy to check that $\ovGa$ is a weak $(s, r)$-metacirculant relative to $(\mu, \nu)$, so the assertions for
line 1 of Table~\ref{tbl:wm} all hold. The graph $\ovGa$ is the graph $X_o(s,r;1)$ defined in
\cite[Example 2.1]{MS} (although in \cite{MS} both of $r,s$ are assumed to be odd).
If indeed both of $r, s$ are odd then $\ovGa$, and hence also $\Ga$, is $\frac{1}{2}$-arc transitive
(see \cite[Theorem 4.1]{MS} and its proof).

Now consider line 2 of Table~\ref{tbl:wm}.
Then $\ovG=\la \mu^2, \mu\nu,\sigma\ra$ where again $\mu, \nu, \sigma,$ are as in Definition~\ref{def1}.
This time $\ovGa$ is a weak $(s, r/2)$-metacirculant relative to $(\mu^2, \mu\nu)$, so the assertions for
line 2 of Table~\ref{tbl:wm} all hold. The graph $\ovGa$ is the graph $X_e(s,r/2;1,0)$ defined in
\cite[Example 2.2]{MS} (although in \cite{MS}, $r/2$ is assumed to be even and at least $4$).
Thus Corollary~\ref{cor:wm} is proved.

\end{document}